\documentclass[reqno,11pt]{amsart}

\usepackage{amssymb,amsmath,amsthm,amsfonts,mathtools,bbm,mathrsfs,enumitem,graphics,float,multirow,enumitem,cases}
\usepackage[normalem]{ulem}
\usepackage{hyperref}
\usepackage{comment}
\usepackage{hyperref}
\hypersetup{
    colorlinks,
    linkcolor={blue!80!black},
    citecolor={red!80!black},
    urlcolor={red!80!black}
}
\usepackage{xargs}                      
\usepackage{xcolor}  
\usepackage[colorinlistoftodos,prependcaption,textsize=tiny]{todonotes}
\newcommandx{\unsure}[2][1=]{\todo[linecolor=red,backgroundcolor=red!25,bordercolor=red,#1]{#2}}
\newcommandx{\change}[2][1=]{\todo[linecolor=blue,backgroundcolor=orange!25,bordercolor=blue,#1]{#2}}
\newcommandx{\info}[2][1=]{\todo[linecolor=OliveGreen,backgroundcolor=OliveGreen!25,bordercolor=OliveGreen,#1]{#2}}

\usepackage[top=1.6in,bottom=1.5in,left=1.3in,right=1.3in]{geometry}
\allowdisplaybreaks
\usepackage{cellspace}
\newcommand\redout{\bgroup\markoverwith{\textcolor{red}{\rule[0.5ex]{2pt}{0.8pt}}}\ULon}

\newtheorem{theorem}{Theorem}
\newtheorem{lemma}[theorem]{Lemma}

\newtheorem{remark}{Remark}[section]

\usepackage{tikz}
\usetikzlibrary{calc}

\tikzset{font= }

\newcommand\nc\newcommand
\nc\bfa{{\boldsymbol a}}\nc\bfA{{\boldsymbol A}}\nc\cA{{\mathscr A}}
\nc\bfb{{\boldsymbol b}}\nc\bfB{{\boldsymbol B}}\nc\cB{{\mathscr B}}
\nc\bfc{{\boldsymbol c}}\nc\bfC{{\boldsymbol C}}\nc\cC{{\mathscr C}}
\nc\bfd{{\boldsymbol d}}\nc\bfD{{\boldsymbol D}}\nc\cD{{\mathscr D}}
\nc\bfe{{\boldsymbol e}}\nc\bfE{{\boldsymbol E}}\nc\cE{{\mathscr E}}
\nc\bff{{\boldsymbol f}}\nc\bfF{{\boldsymbol F}}\nc\cF{{\mathscr F}}
\nc\bfg{{\boldsymbol g}}\nc\bfG{{\boldsymbol G}}\nc\cG{{\mathscr G}}
\nc\bfh{{\boldsymbol h}}\nc\bfH{{\boldsymbol H}}\nc\cH{{\mathscr H}}\nc\fH{{\mathfrak H}}
\nc\bfi{{\boldsymbol i}}\nc\bfI{{\boldsymbol I}}\nc\cI{{\mathcal I}}
\nc\bfj{{\boldsymbol j}}\nc\bfJ{{\boldsymbol J}}\nc\cJ{{\mathscr J}}
\nc\bfk{{\boldsymbol k}}\nc\bfK{{\boldsymbol K}}\nc\cK{{\mathscr K}}
\nc\bfl{{\boldsymbol l}}\nc\bfL{{\boldsymbol L}}\nc\cL{{\mathscr L}}
\nc\bfm{{\boldsymbol m}}\nc\bfM{{\boldsymbol M}}\nc\cM{{\mathscr M}}
\nc\bfn{{\boldsymbol n}}\nc\bfN{{\boldsymbol N}}\nc\sN{{\mathscr N}}
\nc\bfo{{\boldsymbol o}}\nc\bfO{{\boldsymbol O}}\nc\cO{{\mathscr O}}
\nc\bfp{{\boldsymbol p}}\nc\bfP{{\boldsymbol P}}\nc\cP{{\mathscr P}}
\nc\bfq{{\boldsymbol q}}\nc\bfQ{{\boldsymbol Q}}\nc\cQ{{\mathscr Q}}
\nc\bfr{{\boldsymbol r}}\nc\bfR{{\boldsymbol R}}\nc\cR{{\mathscr R}}
\nc\bfs{{\boldsymbol s}}\nc\bfS{{\boldsymbol S}}\nc\cS{{\mathscr S}}
\nc\bft{{\boldsymbol t}}\nc\bfT{{\boldsymbol T}}\nc\cT{{\mathscr T}}
\nc\bfu{{\boldsymbol u}}\nc\bfU{{\boldsymbol U}}\nc\cU{{\mathscr U}}
\nc\bfv{{\boldsymbol v}}\nc\bfV{{\boldsymbol V}}\nc\cV{{\mathscr V}}
\nc\bfw{{\boldsymbol w}}\nc\bfW{{\boldsymbol W}}\nc\cW{{\mathscr W}}
\nc\bfx{{\boldsymbol x}}\nc\bfX{{\boldsymbol X}}\nc\cX{{\mathscr X}}
\nc\bfy{{\boldsymbol y}}\nc\bfY{{\boldsymbol Y}}\nc\cY{{\mathscr Y}}
\nc\bfz{{\boldsymbol z}}\nc\bfZ{{\boldsymbol Z}}\nc\cZ{{\mathscr Z}}
\nc\pp{\mathbb{P}}
\nc\ee{\mathbb{E} }

\renewcommand{\le}{\leqslant}
\renewcommand{\leq}{\leqslant}
\renewcommand{\ge}{\geqslant}
\renewcommand{\geq}{\geqslant}

\DeclareMathOperator{\Var}{Var}
\nc{\Cay}{{\sf Cay}}

\nc{\ff}{{\mathbb F}}

\newcommand\remove[1]{}

\title[]{Bootstrap percolation on a generalized Hamming cube}
\author[]{Fengxing Zhu}\thanks{Institute for Systems Research and Department of ECE, University of Maryland, College Park, MD 20742, USA, fengxing@terpmail.umd.edu. Supported in part by NSF grant CCF 2330909.}
\date{}

\begin{document}
\begin{abstract}
We consider the $r$-neighbor bootstrap percolation process on the graph with vertex set $V=\{0,1\}^n$ and edges connecting the pairs at Hamming distance $1,2,\dots,k$, where $k\ge 2$. We find asymptotics of the critical probability of percolation for $r=2,3$. In the deterministic setting, we obtain several results for the size of the smallest percolating set for $k\ge 2$.
\end{abstract}
\maketitle

\section{Introduction}
The process of $r$-neighbor bootstrap percolation on an undirected graph $G(V,E)$ was introduced by Chalupa, Leith, and Reich \cite{Chalupa_1979}. To initialize the process, a subset of vertices $\cA=\cA_0\subset V$ are designated as ``infected''. One step
of bootstrap percolation infects a healthy vertex if it has $\ge r$ neighbors in $\cA_0$. This can potentially augment the subset $\cA_0$
to $\cA_1$ by adding new infected vertices. The status of the infected vertices never changes back, and in the next step 
there may emerge new vertices in $\cA_1^c$ with $r$ or more neighbors in $\cA_1$. They in turn become infected, and the process
proceeds until it either cannot advance any further because of the lack of well-connected vertices or because the entire $V$
has become infected. In the latter case we say that, with the initial choice of $\cA_0$, percolation takes place and we call the set
$\cA_0$ {\em contagious}. To state this
formally, for all $i\ge 1$ define the sets
    $$
    \cA_i = \cA_{i-1} \cup \{v \in V:|N_v \cap \cA_{i-1}| \geq r \}
    $$
and say that the process stops if for some $i$, $\cA_i=\cA_{i-1}$ or $\cA_i=V$ (percolation occurs).
Here $N_v=N_v(G)=\{u\in V: uv\in E\}$ is the neighborhood of $v$ in $G$. The parameter $r$ is called the {\em infection threshold}.

Bootstrap percolation can be studied in the deterministic or random setting depending on the choice of the initializing set $\cA_0$.
In random bootstrap percolation, the set $\cA_0$ is formed by randomly and independently placing each $v\in V$ with probability $p$.
The central question in the problem of random bootstrap percolation is to determine the critical probability, defined as
   $$
   p_c(G,r)=\sup\{p \in (0,1): \mathbb{P}_p(\cA_0 \; \text{percolates on} \: G) \leq \frac{1}{2} \}.
   $$
   
Among the graphs for which $p_c$ has been extensively studied, are the Boolean hypercube $Q_n$ (the Hamming graph) and the grid graph.  The authors of \cite{balogh_bollobas_2006}
obtained a tight estimate for the critical probability $p_c(Q_n,2)$ up to a constant factor, with
a sharper estimate derived subsequently in \cite{balogh_bollobas_morris_2010}. In addition to examining the critical probability behavior when the infection threshold is constant, \cite{balogh_bollobas_morris_2009} provided a sharp estimate for $p_c(Q_n,\frac{n}{2})$ and derived the exact second-order term.

Considerable effort has also been devoted to studying the critical probability on $d$-dimensional grid graphs $[n]^d$. We mention \cite{Aizenman_1988,Cirillo,CERF200269,Holroyd2002SharpMT,Balogh2},
with the final result of \cite{Balogh} establishing a sharp estimate for $p_c([n]^d,r)$, $2 \leq r \leq d$.


In the deterministic setting one of the main problems addressed is the size and structure of the smallest ``contagious set'', i.e., the
initializing set that leads to infecting the entire $G$ by the $r$-process. Another related question is which contagious
sets take the most time to percolate, i.e., finding the maximum number of steps of the process until percolation, assuming that $\cA_0$ is contagious. Let $m(G,r)$ denote the size of a minimum contagious set and let $T(G,r)$ denote the maximum percolation time on a graph $G$ with the infection threshold $r$. For the hypercube, the authors of \cite{MORRISON201861} obtained a tight estimate of $m(Q_n,r)$. Paper \cite{Michal} established the exact value of $T(Q_n,2)$ and \cite{Ivailo} gave an estimate of $T(Q_n,r \geq 3)$. For the grid graph, \cite{dukes2023extremalboundsthreeneighbourbootstrap}
found the value of $m([a_1] \times [a_2] \times [a_3],3)$ for all but small values of 
the dimensions $a_i$. A characterization of $m([n]^2,3)$ was obtained in \cite{Fabricio,dukes2023extremalboundsthreeneighbourbootstrap}.
Additionally, \cite{Fabricio2} determined the exact maximum percolation time on $[n]^2$ with an infection threshold of $2$ for all contagious sets of minimum size, and \cite{Benevides} found the exact first order term of $T([n]^2,2)$.

While studies of bootstrap percolation have primarily focused on the hypercube and grid graphs, several works were devoted to the percolation process on other graphs. For instance, \cite{FREUND201866} found the value of $m(G,2)$ when $G$ is an Ore graph. 
The authors of \cite{amin} estimated the size of minimum contagious sets for expander graphs. Further, \cite{holroyd} modified the grid graph by allowing
edges between vertices as distance $k\ge 1$ in both horizontal and vertical directions.

Thinking of a general graph $G(V,E)$, it is possible to phrase the bootstrap percolation process as a question in coding theory. Namely, assign to each vertex a symbol from a finite alphabet, say 0 or 1, and consider a code $C(G)$ formed of such assignments, i.e., of binary vectors $c\in\{0,1\}^{|V|}$. Suppose moreover that for every vector $c\in C(G)$, the bit assignment of every vertex $v$ can be uniquely recovered from the values of $r$ of its neighbors in $G$. Having observed a vector $c$ with a subset of coordinates erased, we aim at recovering their values from the nonerased coordinates. This may require several rounds of recovery, which exactly corresponds to the bootstrap percolation process on $G$. This interpretation was first observed in \cite{barg2022high}.

Inspired by these results, in this paper we study bootstrap percolation on the hypercube with edges added between the pairs of vertices at Hamming distance $k\ge 2$. We denote this graph by $Q_{n,k}(V,E),$ where $V=\{0,1\}^n$ and $E=\{xy:1\le d_{yx}\le k\}$, where $d_{xy}:=d_H(x,y)$ is the Hamming distance and $k$ is a fixed integer. Given a subgraph $G\subset Q_{n,k}$ and $v\in V$, we denote by $N_v(G)$ the set of neighbors of $v$ in $G$, omitting the 
reference to $G$ if it is the entire cube.

\subsection{Main Results}
The first two theorems in our main results concern the critical probability $p_c(Q_{n,k},r)$ for $r=2,3$, for which we obtain a complete characterization. We now state these results formally.

\begin{theorem} \label{thm:r=2} (Critical probability for the 2-neighbor process on $Q_{n,k}$)

Let $c >0$ be a constant, $r=2$, and $k \geq 2$. As $n \rightarrow \infty$, we have 
\begin{align*}
 \mathbb{P}(\cA_0 \; \text{percolates}) \rightarrow
    \begin{cases}
  1& \text{if $p \gg 2^{-\frac{n}{2}} n^{-k}$} \\
  0 & \text{if $p \ll 2^{-\frac{n}{2}} n^{-k}$} \\
   1-\exp \left(-\frac{c^2}{2(2k)!} \right)& \text{if $p=c 2^{-\frac{n}{2}} n^{-k}$}.
\end{cases}
\end{align*}
Furthermore, if $p=c2^{-\frac{n}{2}} n^{-k}$, then the number of unordered pairs $(x,y)$ such that $d_{xy}\leq 2k$ is asymptotically Poisson distributed with mean $\frac{c^2}{2(2k)!}$.
\end{theorem}

\begin{theorem} \label{thm: r=3} (Critical probability for the 3-neighbor process on $Q_{n,k}$)

Let $c >0$ be a constant, $r=3$, and $k \geq 2$. As $n \rightarrow \infty$, we have 
\begin{align*}
 \mathbb{P}(\cA_0 \; \text{percolates}) \rightarrow
    \begin{cases}
  1& \text{if $p \gg 2^{-\frac{n}{3}} n^{-k}$} \\
  0 & \text{if $p \ll 2^{-\frac{n}{3}} n^{-k}$} \\
   1-\exp\left(-\frac{c^3}{6\binom{2k}{k}(2k)! k!} \right)& \text{if $p=c 2^{-\frac{n}{3}} n^{-k}$}.
\end{cases}
\end{align*}
Furthermore, if $p=c2^{-\frac{n}{3}} n^{-k}$, then the number of unordered triples $(x,y,z)$ such that $d_{xy}\leq 2k$, $d_{xz} \leq 2k$ and $d_{yz} \leq 2k$ is asymptotically Poisson distributed with mean $\frac{c^3}{6\binom{2k}{k}(2k)! k!}$.
\end{theorem}

First, observe that the above two theorems imply that the critical probabilities for the $2$-neighbor and $3$-neighbor processes on $Q_{n,k}$ do not exhibit sharp thresholds. This stands in contrast to the $2$-neighbor process on $Q_n$, for which a sharp threshold\footnote{An event A has a sharp threshold if the window (in $p$) in which A has a probability between $\epsilon$ and $1-\epsilon$ has size $o(p_c)$; otherwise it has a coarse threshold.} was established in \cite{balogh_bollobas_morris_2010}.

The characterizations of the critical probability for $r=2,3$ appear to be analogous to the threshold functions for the appearance of a fixed strictly balanced subgraph in $G(n,p)$, the random graph on $n$ vertices in which each pair of vertices is joined by an edge independently with probability $p$. The main difference is that, in our setting, one must first obtain a characterization of the percolating sets and then apply standard probabilistic tools such as the second moment method, the FKG inequality, the exponential tail bounds of Janson, {\L}uczak, and Ruci{\'n}ski, and $k$th moment calculations to prove convergence to a Poisson distribution.

In the case $k=2$, it is easy to show that if the initially infected set contains a pair $(x,y)$ such that $d_{xy} \leq 2k$ then the every vertex on the graph will eventually be infected. With this characterization of percolating sets, the expectation value of the number of such unordered pairs will give rise to the critical probability via the second moment method. However, in the case $k=3$, establishing an analogous statement is more involved. 

Now we will state the two theorems in our main results concerning the size of minimum percolating sets. 
\begin{theorem}\label{thm:max-set} The size of the smallest percolating set $m(r)$ satisfies
\begin{align}
    {m(r)}& =r       \text{\quad \quad if $k=2$ and $ 2 \leq r \leq 6$} \label{eq:2-6}\\
    m(r) & \leq \binom{r}{m_r}     \text{\quad if $k=2$ and $ 10 \leq r \leq n$} \label{eq:8-n}\\
    m(r) & \leq \binom{n+1}{m_r}  \text{\quad if $k=2$ and $n \leq r \leq \frac{n^2}{8}$} \label{eq:>n} \\
       m(r) & \leq \sum_{j=0}^{k-1} \binom{r-1}{m_{k,r}+j}  \text{\quad \quad if $k\geq 3$, $r=cn$ with $0<c<1$, and $r \geq 2^{\frac{k^2}{k-1}}$}   \label{eq: r=cn}   \\
    m(r) & \leq \sum_{j=0}^{k-1} \binom{n}{m_{k,r}+j} \text{\quad \quad if $k\geq 3$ and $ n \leq r \leq \frac{n^k}{(2k)^k}$}  \label{eq: r less} \\
    m(r) & >r \text{\quad \quad \quad if $k=2$ and $r \geq 7$} \label{eq:>15}\\
    m(r) & =r \text{\quad \quad \quad if $k\geq 3$ and $r=o(n)$,} \label{eq:>3}
  \end{align}
where $m_r:=\lfloor\sqrt {2r}\rfloor$ and $m_{k,r}=\lfloor kr^{1/k}\rfloor$. 
\end{theorem}

\begin{theorem} \label{thm: min-set}
    Let $n$ be sufficiently large, and let $0 <\delta <1$, $0<\ell<1$, and $0 <c<1$. We have
\begin{align}
    {m(r)}& \geq \delta 2^{\sqrt{2cr^{\ell}}}    \text{\quad \quad if $k=2$ and $   r \geq n^{\frac{2}{2-\ell}}\left(\frac{\sqrt{2c}}{2(1-c)} \right)^{\frac{2}{2-\ell}}$} \label{eq:k=2}  \\
    {m(r)} & \geq \delta 2^{\sqrt{2cr}} \text{\quad \quad if $k=2$  and $  r \geq \frac{n^2}{\left(\sqrt{2c}+\frac{(1-c)}{\sqrt{2c}}\right)^2}$} \label{eq: k=2,l=1}\\
    m(r) & \geq \delta \exp \left(r^{\ell/k}\frac{k}{e}(\frac{c}{k})^{1/k}\right)  \text{\quad \quad if $k\geq 3$ and $r \geq n^{\frac{k(k-1)}{k-\ell}} \left(\frac{e}{k-1}\right)^{\frac{(k-1)k}{k-\ell}}\left(\frac{c}{2}\right)^{\frac{k}{k(k-\ell)}\frac{1}{1-c}}$}   \label{eq: k>=3,l<1}  \\ 
    {m(r)} & \geq  \delta \exp \left({r^{1/k} \frac{k}{e} \left(\frac{c}{2}\right)^{1/k}} \right)  \text{\quad \quad if $k \geq 3$ and $r \geq n^k \left(\frac{2}{1-c}\left(\frac{e}{k-1} \right)^{k-1} \left(\frac{c}{2} \right)^{1/k}\frac{k}{e}\right)^{\frac{k}{k-1}}$.} \label{eq: k>=3}
\end{align}

\end{theorem}

For the upper bound on the minimum size of a percolating set, it suffices to construct a small set of vertices whose infection eventually leads to full infection. Establishing a lower bound, however, is much more elusive for general $r$. In this setting, the linear-algebraic method introduced by Balogh, Bollob{\'a}s, Morris, and Riordan \cite{Balough2012} is essentially the only general tool currently available. However, we encountered some difficulty in applying this method to our graph, and therefore resorted to a more ad hoc argument to obtain a meaningful lower bound.

The organization of the paper is as follows. In Section~\ref{sec:pre}, we present and prove several preliminary results and tools. Theorems~\ref{thm:r=2} and~\ref{thm: r=3} are proved in Sections~\ref{sec:p2} and~\ref{sec:p3}, respectively. In Section~\ref{sec:MPS}, we prove Theorems~\ref{thm:max-set} and~\ref{thm: min-set}.

\section{Preliminary Results and Tools} \label{sec:pre}
Let us begin with two simple lemmas used repeatedly in the paper.

\begin{lemma} \label{lemma:subcube}
Let $r,k\ge 2, n \geq r-1$ and let $\cA_0\subset Q_{n,k}$ be a set of already infected vertices. If $\cA_0$ contains an $(r-1)$-dimensional subcube of $Q_{n,k}$, then the $r$-process on $Q_{n,k}$ will percolate after several additional steps. 
\end{lemma}
\begin{proof}   
W.l.o.g., consider the subcube\footnote{by abuse of notation, we may write $Q_{r-1,k}$ to refer to an $(r-1)$-dimensional subcube of $Q_{n,k}$, where the locations of $\ast$ are not necessarily consecutive.} $Q_{r-1,k}^{0}=\{x=(\ast^{r-1}0^{n-r+1})\}$ formed of all the vertices that end with $n-r+1$ 0's and assume that all of its vertices are infected. 
We will show by induction that all the subcubes $Q_{i,k}^{0}, i=r,\dots, n$ become infected,
implying the percolation claim. For the base, consider the vertex $v=(0^{r-1}10^{n-r})$ and observe that $|N_v(Q_{r-1,k}^{0})|\ge r$ for all $k\ge 2$: the origin $0^{n}$ and all the vertices
of the form $x=(z|0^{n-r+1})$ with $|z|=1$ satisfy $d(v,x)\le 2$, so the infection spreads to $v$.
Similarly, all the vertices of the form $x=(\ast^{r-1}|1|0^{n-r})$ have $\ge r$ neighbors in the set
$Q_{r-1,k}^{0}$, so the entire set $Q_{r,k}^0$ gets infected, completing the base step.
The transition from $Q_{i-1,k}^0$ to $Q_{i,k}^0,r+1\le i\le n$ is argued in the same way (in fact, it becomes easier to propagate the infection as the dimension grows because the set of
close neighbors in $Q_{i-1,k}^0$ grows in size as $i$ increases).
\end{proof}

Because of this observation, to prove that percolation takes place, it suffices to show that with the appropriate
choice of the initially infected set $\cA_0\subset V$ several steps of the process will infect a subcube $Q_{r-1,k}$.
We will use variations of this approach in the proofs of the main results.

\begin{lemma} \label{lemma15} Let $X_n, n\ge 1$ be a collection of random variables such that $X_n \in {\mathbb N}_0$, 
$EX_n\to\infty$ and $\Var(X_n)=o(\ee X_n^2)$ as $n\to\infty$. Then $\pp (X_n=0)=o(1)$ and 
   $$
   \pp\Big(\frac{\ee X_n}{2} \leq X_n \leq \frac{3}{2}\ee X_n\Big)=1-o(1).
   $$
\end{lemma}
\begin{proof} The first statement follows by an application of the second moment method and the second one 
by Chebyshev's inequality.
\end{proof}

Another tool that we need is the FKG-inequality. Before stating the FKG-inequality we need to introduce some concepts and notation. 

Suppose that $\Omega = \{0, 1\}^{2^n}$, and each coordinate is assigned independently. In other words, we
have a collection of independent Bernoulli random variables $X_1, . . . , X_{2^n}$.

Define a partial order on the elements of $\Omega$ as follows: 
$$(x_1, . . . , x_{2^n}) \geq (y_1, . . . , y_{2^n})$$ 
if and only if $x_i \geq y_i$ for all $1 \leq i \leq 2^n$.

We say that an event $A \subset \Omega$ is increasing(decreasing) if $x \in A$ and $y \geq x(y \leq x)$ implies that $y \in  A$. 

There is a one-to-one correspondence between the initial infection configuration and $\Omega$. Let us enumerate all the vertices in $V(Q_n)$ as $v_1,v_2,\cdots,v_{2^n}$ and denote $x_i=1$ if $v_i$ is initially infected and otherwise $x_i=0$.

\begin{lemma} (FKG inequality \cite{FKG}) \label{lemma: FKG}
In the above setting, if both events $A$ and $B$ are increasing (or decreasing), then 
$$\mathbb{P}(A \cap B) \geq \mathbb{P}(A) \mathbb{P}(B).$$
\end{lemma}

 Another tool required to derive our main theorems is the following inequality by Janson, {\L}uczak, and Ruci{\'n}ski. 

\begin{lemma} (Janson,{\L}uczak, and Ruci{\'n}ski \cite{Random_graphs}) \label{lemma: Janson}
Let $A_1,\cdots, A_m \subset [n]$. Let $R \subset [n]$ be chosen randomly such that $$\mathbb{P}(i \in R)=p$$ 
for each $i \in [n]$ independent of $ j \neq i \in [n]$. Define events $\{B_j\}_{j=1}^m$ to be $$B_j=\{A_j \subset R\}.$$ 
Let $\mu= \mathbb{E}[\sum_{j=1}^m \mathbb{I}_{B_j}]$ and $\Delta=\sum_{i \sim j} \mathbb{P}(B_i \cap B_j)$ where $i \sim j$ if $A_i \cap A_j \neq \emptyset$ for $i \neq j$. Then we have 

\begin{align*}
 \mathbb{P}(\cap_{j=1}^m B_j^c) \leq 
    \begin{cases}
  \exp(-\mu +\Delta)& \text{if $\Delta \leq \mu$} \\
   \exp(-\frac{\mu^2}{\mu+2\Delta}) & \text{if $\Delta \geq \mu$.} \\
\end{cases}
\end{align*}
\end{lemma}
Note that $\Delta$ excludes the diagonal terms with $i=j$. 

The main theorem for proving convergence to the Poisson distribution is the following. 

\begin{theorem} (Frieze,Kro{\'n}ski \cite{Random_graphs_Frieze2}) \label{thm:random graphs}
Let $S_n=\sum_{i\ge 1}^n I_i$, be a sequence of random variables $n \geq 1$, and define the binomial moments $B_k^{(n)}=\mathbb{E}\binom{S_n}{k}$. Suppose that there exists $\lambda\ge 0$ such that, for every fixed $k\ge 1$,
\[
\lim_{n\to\infty} B_k^{(n)}=\frac{\lambda^k}{k!}.
\]
Then, for every $j\ge 0$,
\[
\lim_{n\to\infty}\mathbb{P}(S_n=j)=e^{-\lambda}\frac{\lambda^j}{j!},
\]
i.e., $S_n$ converges in distribution to a $\mathrm{Poisson}(\lambda)$ random variable.

\end{theorem}

\begin{remark} (Frieze, Kro{\'n}ski \cite{Random_graphs_Frieze2}) \label{rmk:random graphs}
   Notice that the falling factorial
\[
(S_n)_k = S_n(S_n - 1)\cdots(S_n - k + 1)
\]
counts the number of ordered $k$-tuples of events with $I_i = 1$. Hence the binomial moments of $S_n$ can be replaced in Theorem~\ref{thm:random graphs} by the factorial moments, defined as
\[
\mathbb{E}(S_n)_k
=
\mathbb{E}\big[ S_n(S_n - 1)\cdots(S_n - k + 1) \big],
\]
and one has to check whether, for every $k \ge 1$,
\[
\lim_{n\to\infty} \mathbb{E}(S_n)_k = \lambda^k.
\] 
\end{remark}

\section{Critical probability for the 2-neighbor process} \label{sec:p2}
In this section we will prove Theorem ~\ref{thm:r=2}.

Below by $\omega(n)$ and $C(k)$ we denote an arbitrarily slowly growing function of $n$ and a constant which may depend on $k$ respectively. We will use $\omega(n)$ and $C(k)$ as generic notations, where the specific form of these functions may be different in different expressions. 

Let us prove the following Lemma.
\begin{lemma} \label{lemma:r=2}
As $n \rightarrow \infty$, we have 
    \begin{equation}\label{eq:pc22}
    \frac{1}{\omega(n)}2^{-\frac{n}{2}} n^{-k} \leq p_c(Q_{n,k},2) \leq  \omega(n)2^{-\frac{n}{2}} n^{-k}.
    \end{equation}
\end{lemma}


\begin{proof}
We assume that at the start of the process, every vertex is infected independently of the others with probability $p$ and is not 
infected with probability $1-p$. Let $\cA_0$ be the set of initially infected vertices. We claim that percolation occurs if and only if there is a pair $x,y\in \cA_0$ with $d_{xy}\le 2k$. If not, then no vertex will have 2 neighbors in $Q_{n,k}$,
so the process cannot evolve. Now suppose the assumption holds true. We will prove that the process percolates,
differentiating between the following two cases.

Suppose that $d_{xy}=t,$ where $1\le t\le k$. Let us write $x=(x_1,\dots,x_t,x_{t+1},\dots x_n)$ and 
$y=(\bar{x}_1 \bar{x}_2\dots \bar{x}_t x_{t+1}\dots x_n)$ where $\bar{x}$ refers to the negation of the bit $x$.
Making one step of the process ensures that all the vertices in the subcube $(**\dots *|x_{t+1}x_{t+2}\dots x_n)$ are infected,
and the claim follows by Lemma~\ref{lemma:subcube}.

Now suppose that $k+1 \leq t \leq 2k$. Writing $x$ and $y$ as above, consider the vertex
$u=(x_1x_2\dots x_k\bar{x}_{k+1}\bar{x}_{k+2}\dots \bar{x}_tx_{t+1}\dots x_n)$. We have $d_{ux}=t-k\le k$ and $d_{uy}=k,$
so after one step, $u$ will be infected. Considering the pair $(u,x),$ we have reduced the percolation task to the previous case, so our claim is proved.

Now let $X_i,i=1,2,\dots, 2k$ denote the number of initially infected pairs $(x,y)$ with $d_{xy}=i$. Clearly, 
  $$
  \ee \sum_{i=1}^{2k}X_i=\sum_{i=1}^{2k}2^{n-1}\binom{n}{i}p^2. 
  $$
For a pair $T=(x,y)\in V\times V$, let $X_{T}={\mathbb I}_{\{T\text{ initially infected}\}}$.
Let $\cT_i:=\{(x,y)\in V\times V: d_{xy}=i\}$.
     \begin{equation}\label{eq:var}
     \Var\Big(\sum_{i=1}^{2k}X_i\Big)=\sum_{i<j}\sum_{T\in \cT_i,T'\in \cT_j}\text{Cov}(X_T,X_{T'})
     +\sum_{i=1}^{2k}\sum_{T ,T'\in \cT_i}\text{Cov}(X_T,X_{T'}).
     \end{equation}
     Let $T \in \cT_i$ and $T' \in \cT_j$ with $i<j$. Then we have 
        \begin{align*}
\text{Cov}(X_T,X_{T'}) &=\ee [X_TX_{T'}]-\ee [X_T] \ee  [X_{T'}]\\
&=p^{|T \cup T'|}-p^{|T|+|T'|}\\
&=
\begin{cases}
  0& \text{if $|T \cap T'|=0$} \\
  p^3-p^4 & \text{if $|T \cap T'|=1$}.
\end{cases}.
     \end{align*}
Summing over all choices of the pairs $T_i,T'_i$, we obtain
\begin{align}
\sum_{T \in T_i.T' \in T_j} \text{Cov}(T,T')= \binom{n}{i} \binom{n}{j} (p^3-p^4)2^n. \label{eq:ij}
\end{align}
Now let $T,T' \in \cT_i$. We have
    \begin{align*}
\text{Cov}(X_T,X_{T'})& =\ee [X_TX_{T'}]-\ee [X_T] \ee  [X_{T'}]\\
&=p^{|T \cup T'|}-p^{|T|+|T'|}\\
&=
\begin{cases}
  p^2-p^4& \text{if $T=T'$} \\
  p^3-p^4 & \text{if $|T \cap T'|=1$} \\
  0 & \text{if $|T \cap T'|=0$}.
\end{cases}
\end{align*}
Summing over all choices of the pairs $T_i,T'_i$, we obtain
\begin{align}
\sum_{T, T' \in \cT_i} \text{Cov}(T,T')=(p^2-p^4)2^n \binom{n}{i}+
(p^3-p^4)2^n \binom{n}{i}\Big(\sum_{l=0}^{i-1}\binom{i}{l} \binom{n-i}{i-l} \Big). \label{eq:ii}
\end{align}

Let $X=\sum_{i=1}^{2k}X_i$ (the dependence on $n$ is suppressed). Therefore,
   $$
   \ee X =C(k)2^nn^{2k}p^2.
   $$
To fit the assumptions of Lemma~\ref{lemma15}, we will choose $p$ so that $\Var(X)=o(\ee X^2)$. Write the variance using \eqref{eq:ij} and \eqref{eq:ii} in \eqref{eq:var} to observe that
  $$
  \text{Var} \left(\sum_{i=1}^{2k}X_i \right) \leq C(k)(p^22^nn^{2k}+p^32^nn^{4k}).
  $$
It suffices to take 
 $$
 (p^32^nn^{4k})^{\frac{1}{2}}\ll 2^nn^{2k}p^2,
 $$ 
or equivalently
   $$p \gg 2^{-n},$$
and 
   $$
   (p^22^nn^{2k})^{\frac{1}{2}}\ll 2^nn^{2k}p^2,
   $$ 
or equivalently 
   $$ 
   p \gg 2^{-\frac{n}{2}}n^{-k}.
   $$
The second of these conditions is more restrictive, giving the lower bound in the claim \eqref{eq:pc22} of the theorem. 

At the same time,
by Markov's inequality, if $p \ll 2^{-\frac{n}{2}}n^{-k}$, then $\ee X\to 0$ if $n\to\infty$, and thus, percolation
does not occur with probability approaching 1. This proves the upper bound in \eqref{eq:pc22}.
\end{proof}

The following lemma shows that if $p=c2^{-\frac{n}{2}}n^{-k}$, then the probability of full infection converges to a constant. This is analogous to the behavior of threshold functions for the appearance of a strictly balanced subgraph in $G(n,p)$. 
\begin{lemma}
Let $c >0$ be a constant and $p=c 2^{-\frac{n}{2}} n^{-k}$. As $n \rightarrow \infty$, we have 
\begin{align*}
 \mathbb{P}(\cA_0 \; \text{percolates}) \rightarrow 1- \exp \left(-\frac{c^2}{2(2k)!} \right). 
\end{align*}
\end{lemma}
\begin{proof}
Define  
\begin{equation*}
  \mathbb{I}_{i}=  
  \begin{cases}
   1 & \text{if $i$th pair $(x,y)$ with $d_{xy}\leq 2k$ is initally infected} \\
  0 & \text{otherwise}.
\end{cases}.
\end{equation*}
Define $N=\sum_{i=1}^{2k}2^{n-1}\binom{n}{i}$. Then we have 
    \begin{align}
    \mathbb{P}(\cA_0 \; \text{does not percolates})
    & = \mathbb{P}(\sum_{i=1}^{N}\mathbb{I}_{i}=0)\\
    & = \mathbb{P}(\cap_i^{N} \{\mathbb{I}_{i}=0\})\\
    & \geq \prod_{i=1}^{N}\mathbb{P}(\{\mathbb{I}_i=0\}). \label{eq1: use FKG} \\
    & =(1-p^2)^N\\
    & \rightarrow \exp \left(-\frac{c^2}{2(2k)!}\right)
    \end{align}
Note that the inequality~\ref{eq1: use FKG} is true due to Lemma~\ref{lemma: FKG}. 

Now let us prove the upper bound by using Lemma~\ref{lemma: Janson}.

Let $R$ be the set of initially infected vertices and $\{A_i\}_{i=1}^N$ be all the distinct unordered pairs $(x,y)$ with $d_{xy}\leq 2k$. Let $B_i$ be the event that $A_i \subset R$. 

From the proof of Lemma ~\ref{lemma:r=2}, we have  $$\mu:=\mathbb{E}(\sum_{i=1}^N \mathbb{I}_{B_i})=Np^2$$ and  

$$\Delta:=\sum_{i \sim j}\mathbb{P}(B_i \cap B_j) \leq p^32^nn^{4k}=o(1),$$
where $i \sim j$ if $A_i \cap A_j \neq \emptyset$ for $i \neq j$. 

By Lemma~\ref{lemma: Janson}, we have the desired result. 
\end{proof}
In fact, one can go further and prove that the distribution of the number of unordered pairs $(x,y)$ with $d_{xy} \le 2k$ converges to a Poisson distribution. Analogous results for the number of copies of a strictly balanced subgraph in $G(n,p)$, when $p$ is a constant multiple of the threshold function, were derived by Bollob{\'a}s \cite{Bollobas2}, and  Karo{\'n}ski and Ruci{\'n}ski \cite{Karonski}.

\begin{lemma}
If $p=c2^{-\frac{n}{2}} n^{-k}$, then the number of unordered pair $(x,y)$ such that $d_{xy}\leq 2k$ is asymptotically Poisson distributed with mean $\frac{c^2}{2(2k)!}$.    
\end{lemma}
\begin{proof}
We follow the ideas from Chapter 5 of \cite{Random_graphs_Frieze2} by using Theorem~\ref{thm:random graphs} and Remark~\ref{rmk:random graphs}.  Let $H_1,H_2,\cdots, H_t$ be all the distinct unordered pairs $(x,y)$ with $d_{xy} \leq 2k$ on the graph $Q_{n,k}$. For $i=1,2,\cdots,t$, let 
\begin{align*}
 \mathbb{I}_{H_i} = 
    \begin{cases}
  1 & \text{if $H_i \subset \cA_0$} \\
   0 & \text{otherwise.} \\
\end{cases}
\end{align*}
Let $X_H=\sum_{i=1}^t\mathbb{I}_{H_i}$ and the $\ell$th factorial moment of $X_{H}$, $\ell=1,2,\cdots$, 
$$\mathbb{E}(X_H)_{\ell}=\mathbb{E}[X_H(X_H-1)\cdots(X_H-\ell+1)],$$
can be written as 
\begin{align*}
\mathbb{E}(X_H)_\ell & =\sum_{i_1,i_2,\cdots,i_{\ell}}\mathbb{P}(\mathbb{I}_{H_{i_1}}=1,\mathbb{I}_{H_{i_2}}=1,\cdots,\mathbb{I}_{H_{i_{\ell}}}=1)\\
& =D_{\ell}+\overline{D}_{\ell},
\end{align*}
where the summation is taken over all $\ell$ element sequences of distinct indices $i_j$ from $\{1,2,\cdots,t\}$, and $D_{\ell}$ and $\overline{D}_{\ell}$ denote the partial sums taken over all ordered $\ell$ tuples of unordered pair $(x,y)$ with $d_{xy}\leq 2k$ which are, respectively, pairwise disjoint ($D_\ell$) and not all pairwise disjoint ($\overline{D}_\ell$). 

Now we have 
\begin{align*}
    D_{\ell}& =\sum_{i_1,i_2,\cdots,i_{\ell}} \prod_{j=1}^{\ell}\mathbb{P}(\mathbb{I}_{H_{i_j}}=1)\\
    & =(1+o(1))\left(\frac{1}{2}2^{n}\binom{n}{2k} \right) \left(\frac{1}{2}\left(2^{n}-2 \right)\binom{n}{2k}\right)\cdots \left(\frac{1}{2}(2^{n}-2(\ell-1))\binom{n}{2k} \right)(p^2)^{\ell}\\
    & =(1+o(1))\left(2^{n-1}\binom{n}{2k}\right)^{\ell}p^{2\ell}
\end{align*}
and 
\begin{align*}
\left(\mathbb{E}[X_H]\right)^\ell &=\left(\sum_{i=1}^t2^{n-1}\binom{n}{i}p^2\right)^{\ell}\\
&=(1+o(1))\left(2^{n-1}\binom{n}{2k}\right)^{\ell}p^{2\ell}
\end{align*}
Note that $\left(2^{n-1}\binom{n}{2k}\right)^{\ell}p^{2\ell}=(1+o(1))(\frac{c^2}{2(2k)!})^{\ell}$. If we can show that 
$$\overline{D}_k=o(1),$$
then the proof is finished. 

Let us consider the family $\mathcal{G}_\ell(V,E)$ be the collection of all simple and undirected graphs with at least one edge, where the vertex set $V$ is $\ell$ distinct unordered pairs $(x,y)$ such that $d_{xy}\leq 2k$ and the edge set $E=\{(x_1,y_1)(x_2,y_2): \{x_1,y_1\} \cap \{x_2,y_2\} \neq \emptyset \}$. Note that there is a one-to-one correspondence between the set $\mathcal{G}_\ell(V,E)$ and the set of ordered $\ell$ tuples of unordered pairs $(x,y)$ with $d_{xy}\leq 2k$ which are not all pairwise disjoint. 

Now let us partition the graphs in this collection according to the number of components in the graph. Let $U_{\ell,i}$ denote a the collection of graphs with $i$ components and we have 
$$\mathcal{G}_{\ell}=\bigcup_{i=1}^{\ell-1} U_{\ell,i}. $$
Now let us consider a particular graph $G$ on the collection $U_{\ell,i}$. Since $G$ has $i$ components, it is not difficult to see that the vertex set of $G$ contains at least $2i+1$ distinct elements from $V(Q_{n,k})$. Moreover, it is easy to see that the number of of such ordered $\ell$ tuples of unordered pairs $(x,y)$ with $d_{xy}\leq 2k$ is at most $2^{ni}\binom{n}{2k}^{2l}$ Therefore, we have 
$$\overline{D}_\ell\leq \sum_{i=1}^{\ell-1}\sum_{U_{\ell,i} \in \mathcal{G}_{\ell}}p^{2i+1}2^{ni}n^{4kl} \leq \ell 2^{\binom{\ell}{2}}p^{2i+1}2^{ni}n^{4kl} =o(1),$$
which is the desired result. \qedhere

\end{proof}

\section{Critical probability for the 3-neighbor process} \label{sec:p3}
In this section we will prove Theorem~\ref{thm: r=3}. 

\subsection{The case of {$k=2$}}
We start with $k=2$ and treat general $k$ in the second part of this section.
Thus, in this part we will study $p_c(Q_{n,2},3)$.
The following lemma gives a necessary and sufficient condition on the initially infected set $\cA_0\subset V$
for the $r$-process to percolate. 
\begin{lemma} \label{lemma:triangle} 
$3$-neighbor percolation on $Q_{n,2}$ takes place if and only if there exist vertices $x_1,x_2,x_3\in \cA_0$ with pairwise distances at most $4$.
\end{lemma}
\begin{proof}
To prove the only if part, note that if no three vertices in $\cA_0$ form a triangle with 
sides at most 4, then the process cannot evolve beyond the initially infected set.

To prove sufficiency, suppose that $x_1, x_2,x_3\in {\cA_0}$ are three initially infected vertices. W.l.o.g, we can place $x_1$ at the origin. The proof proceeds by a case study depending on the values of the distances $d_{12},d_{13}$ and $d_{23}$. By inspection, if we disregard the order, there are 9 types of triangles that three vertices can form\footnote{These are all the unordered triples of integers $i,j,k$ with $1\le i,j,k\le 4$ such that $|i-j|\le k\le i+j$ and $i-j+k$ even.}. We will consider each of them, showing that after several steps of the process
there emerges an infected subcube $Q_{2,2}$, and then percolation follows 
from Lemma~\ref{lemma:subcube}.

Case 1: $(d_{12},d_{13},d_{23})=(1,1,2)$.
We already assumed that $x_1=0^n$. W.l.o.g. assume that $x_2=10^{n-1}$, and $x_3=010^{n-2}$. The subcube $Q_{2,2}$ will be formed by vertices $x_1,x_2,x_3,x_4,$ where $x_4=110^{n-2}$. Observe that $x_4$ is added in the next step of the process 
because $x_1,x_2,x_3\in N_{x_{_4}}$ (recall that $N_{v}$  is the neighborhood of $v$ in $Q_{n,k}$). We write this concisely as
   \begin{gather*}
   (x_1,x_2,x_3)\to x_4.
   \end{gather*}

Case 2: $(d_{12},d_{13},d_{23})=(2,2,2)$. Assume that $x_2=110^{n-2}$ and $x_3=0110^{n-3}$. We have $Q_{2,2}=(x_2,x_3,x_4,x_5)$, where 
\begin{gather*}
x_4=010^{n-2}, \quad x_5=1^30^{n-3}.
\end{gather*}
These vertices are infected following the sequence   
  \begin{gather*}
   (x_1,x_2,x_3)\to x_4, \quad (x_2,x_3,x_4)\to x_5.
   \end{gather*}

Case 3: $(d_{12},d_{13},d_{23})=(2,2,4)$. Assume that $x_2=110^{n-2}$ and $x_3=00110^{n-4}$. We have $Q_{2,2}=(x_1,x_3,x_5,x_6)$ where 
\begin{gather*}
x_5=0010^{n-3}, \quad x_6=00010^{n-4}.
\end{gather*}
These vertices  are infected following the sequence
\begin{align*}
&(x_1,x_2,x_3)\to x_4,  \quad (x_1,x_3,x_4)\to x_5,\quad  (x_1,x_3,x_5)\to x_6,
\end{align*}
where $x_4=0110^{n-3}$. 

Case 4:  $(d_{12},d_{13},d_{23})=(1,2,3)$. Assume that $x_2=0010^{n-2}$ and  $x_3=110^{n-2}$. We have $Q_{2,2}=(x_2,x_4,x_5,x_6)$, where
\begin{align*}
&x_4=0110^{n-3}, \quad x_5=1110^{n-3},\\
&x_6=1010^{n-3}.
\end{align*}
These vertices are infected following the sequence
\begin{align*}
 &(x_1,x_2,x_3)\to x_4,  \quad (x_2,x_3,x_4)\to x_5,\quad (x_2,x_4,x_5)\to x_6. 
\end{align*}

Case 5:  $(d_{12},d_{13},d_{23})=(2,3,3)$. Assume that $x_2=110^{n-2}$ and $x_3=01110^{n-4}$. We have $Q_{2,2}=(x_3,x_4,x_5,x_6)$, where
\begin{align*}
&x_4=0110^{n-3}, \quad x_5=1110^{n-3}, \quad x_6=11110^{n-4}.
\end{align*}
These vertices are infected following the sequence
\begin{align*}
&(x_1,x_2,x_3)\to x_4,  \quad (x_2,x_3,x_4)\to x_5,\quad
 (x_3,x_4,x_5)\to x_6. 
 \end{align*}

Case 6: $(d_{12},d_{13},d_{23})=(4,4,4)$. Assume that $x_2=11110^{n-4}$ and $x_3=0011110^{n-6}$. We have $Q_{2,2}=(x_4,x_5,x_6,x_7)$, where 
\begin{gather*}
x_4=00110^{n-4}, \quad x_5=011110^{n-5}, \\
 x_6=001110^{n-5}, \quad x_7=01110^{n-4}.
\end{gather*}
These vertices are infected following the sequence
\begin{gather*}
 (x_1,x_2,x_3)\to x_4,  \quad (x_2,x_3,x_4)\to x_5,\\
 (x_3,x_4,x_5)\to x_6, \quad (x_4,x_5,x_6)\to x_7.  
\end{gather*}

Case 7: $(d_{12},d_{13},d_{23})=(4,2,4)$. Assume that $x_2=11110^{n-4}$ and $x_3=000110^{n-5}$. We have $Q_{2,2}=(x_3,x_4,x_5,x_6)$, where 
\begin{align*}
&x_4=00110^{n-4}, \quad x_5=00010^{n-4},\quad x_6=001110^{n-5}. 
\end{align*}
These vertices are infected following the sequence
\begin{equation}\label{eq:case7}
\begin{aligned}
 &(x_1,x_2,x_3)\to x_4,  \quad (x_1,x_3,x_4)\to x_5, \quad  (x_3,x_4,x_5)\to x_6.
\end{aligned}
\end{equation}

Case 8: $(d_{12},d_{13},d_{23})=(4,1,3)$. Assume that $x_2=11110^{n-4}$ and $x_3=00010^{n-4}$. We have $Q_{2,2}=(x_3,x_4,x_5,x_6)$, where 
\begin{align*}
&x_4=00110^{n-4}, \quad x_5=01010^{n-4},\quad x_6=01110^{n-4}.
\end{align*}
The infection sequence is the same as  in (\ref{eq:case7}).

Case 9: $(d_{12},d_{13},d_{23})=(4,3,3)$. Assume that $x_2=11110^{n-4}$ and $x_3=001110^{n-5}$. We have $Q_{2,2}=(x_3,x_4,x_5,x_6)$, where 
\begin{align*}
&x_4=01110^{n-4}, \quad x_5=00110^{n-4},\quad x_6=011110^{n-5}.
\end{align*}
The infection sequence is again the same as in (\ref{eq:case7}).
\end{proof}

Now we are in a position to find the critical probability for the graph $G$ with 
the infection threshold $r=3$.

\begin{lemma}[Critical probability for the 3-neighbor process with $k=2$]\label{thm:k=2} We have
 \begin{equation}\label{eq:p2k}
 \frac{1}{\omega(n)}2^{-\frac{n}{3}} n^{-2} \leq p_c(Q_{n,k},3) \leq  \omega(n)2^{-\frac{n}{3}} n^{-2}.
 \end{equation}
\end{lemma}

\begin{proof} Call a triple of vertices $x_1,x_2,x_3\in Q_{n,k}$ {\em good} if each of $d_{12},d_{13},d_{23} \leq 4$.
By Lemma~\ref{lemma:triangle}, every vertex will be infected by the $r$-neighbor bootstrap percolation on $Q_{n,k}$ if and only if
the starting set $\cA_0$ contains a good triple. 
Let $X$ be the number of good triples {in $\cA_0$}. To compute $\ee X$, we count the number of triangles whose sides $(d_{12},d_{13},d_{23})$ fall under one of the nine cases enumerated in Lemma~\ref{lemma:triangle}. For instance, the number of triples with $(d_{12},d_{13},d_{23})=(1,1,2)$ is $\frac{2^nn(n-1)}{6}$ since the number of such ordered triple is $2^nn(n-1)$. Indeed, once we fix $x_1$, say to 0, there are $n$ options for $x_2$ and $n-1$ options for $x_3$ to
form the ordered triple. For other values of $(d_{12},d_{13},d_{23})$ the count of triples is found in a similar way, and we obtain 
 \begin{align*}
 \ee [X]  = \frac{1}{6}&p^32^n\Big[n(n-1)+ 2 \binom{n}{2}(n-2) + \binom{n}{2} \binom{n-2}{2}+  n\binom{n-1}{2} \\
      & + 2\binom{n}{2}\binom{n-2}{2}  
   +4\binom{n}{4}\binom{n-4}{2} + 4\binom{n}{4}(n-4)\\
   &+ 4\binom{n}{4} +\binom{n}{4} \binom{4}{2}(n-4)\Big].
 \end{align*}

Let $T_1=\{x_1,x_2,x_3\}$ and $T_1'=\{x_1',x_2',x_3'\}$ be a pair of good triples. For a triple $T$ let $X_T={\mathbbm 1}_{\{T\in \cA_0\}}$.
We have 
\begin{align*}
\text{Cov}(X_{T_1},X_{T_1'})& =\ee [X_{T_1}X_{T_1'}]-\ee [X_{T_1}] \ee  [X_{T_1'}]\\
&=p^{|T_1 \cup T_1'|}-p^{|T_1|+|T_1'|}\\
&=
\begin{cases}
  0& \text{if $|T_1 \cap T_1'|=0$} \\
  p^5-p^6 & \text{if $|T_1 \cap T_1'|=1$}\\
  p^4-p^6 & \text{if $|T_1 \cap T_1'|=2$}\\
  p^3-p^6 & \text{if $|T_1 \cap T_1'|=3$}.\\
\end{cases}
\end{align*}
Estimating the variance, we obtain
\begin{align*}
\text{Var}(X) &=\sum_{T_1,T_1'} \text{Cov}(X_{T_1},X_{T_1'})\\
&  =\sum_{|T_1 \cap T'|=1} \text{Cov}(X_{T_1},X_{T'})+\sum_{|T_1 \cap T'|=2} \text{Cov}(X_{T_1},X_{T'})+\sum_{|T_1 \cap T'|=3} \text{Cov}(X_{T_1},X_{T'})\\
& \leq C \Big\{(p^5-p^6) 2^n \binom{n}{4}^2 \binom{n}{2}^2
+(p^4-p^6)2^n \binom{n}{4}^2\binom{n}{2}\\
&\hspace*{1in} +(p^3-p^6)2^n\binom{n}{4}\binom{n}{2} \Big\}.
\end{align*}
where $C$ is a constant. For the first term above,  the number of good triples $T_1$ is at most $C(k)2^n \binom{n}{4}\binom{n}{2}$ and $T_1'$ is at most $C(k)\binom{n}{4}\binom{n}{2}$ since $|T_1 \cap T_1'|=1$. Other terms can be derived in a similar way. 

In order to satisfy the assumption that $\Var(X)=o((\ee X)^2)$, it suffices to take 
$$p^32^n\binom{n}{4}\binom{n}{2} \gg \left(p^52^n\binom{n}{4}^2\binom{n}{2}^2\right)^{1/2}$$

$$p^32^n\binom{n}{4}\binom{n}{2} \gg \left(p^42^n\binom{n}{4}^2\binom{n}{2}\right)^{1/2}$$
and 
$$p^32^n\binom{n}{4}\binom{n}{2} \gg \left(p^32^n\binom{n}{4}\binom{n}{2}\right)^{1/2}.$$

A direct check shows that 
if $p \gg 2^{-\frac{n}{3}}n^{-2}$, then $\Var(X)=o((\ee X)^2) $ and $\ee X \to \infty$. By Lemma~\ref{lemma15}, 
$\pp(X>0)\to 1$, and then Lemma~\ref{lemma:triangle} implies that percolation occurs, proving the lower bound in 
\eqref{eq:p2k}.
At the same time, if $p\ll 2^{-\frac{n}{3}}n^{-2}$, then by Markov's inequality, $\pp(X \geq 1)\leq \ee [X]=o(1)$. Thus 
with probability $\to 1$ there are no good triples in $\cA_0$, so by Lemma~\ref{lemma:triangle} the process does not percolate.
This proves the upper bound.
\end{proof}
\begin{remark}
    This lemma forms a particular case of Lemma \ref{thm: k>=2}. At the same time, that lemma relies on Lemma~\ref{lemmaV2},
which does not apply to $k=2$, so we kept the presentation of the case $k=2$ self-contained. 
\end{remark}

\subsection{General $k\ge 3$}
We begin with an extension of Lemma~\ref{lemma:triangle} beyond $k=2$.
\begin{lemma} \label{lemmaV2} $3$-neighbor percolation on $Q_{n,k}$ takes place if and only if there exist vertices $x_1,x_2,x_3\in \cA_0$ with pairwise distances at most $2k$. 
\end{lemma}
\begin{proof}
Without loss of generality, assume that $x_1=0^n$,  $x_2=1^{a_1+a_2}0^{n-a_1-a_2}$ and $x_3=0^{a_1}1^{a_2+a_3}0^{n-a_1-a_2-a_3}$.
The distance assumptions translate into $a_i+a_j\le 2k$ with $(i,j)=(1,2), (1,3)$ or $(2,3)$.
The proof proceeds by case analysis in accordance with $a_2\lesseqgtr k$.

\vspace*{.1in}
\paragraph{\bf 1. $a_2 >k$} In this case $a_1<k$ and $a_3 <k$.

\begin{enumerate}[label=1.\arabic*]
    \item $a_1 >0$ and $a_3 >0$. Our plan is to assemble an infected $2$-dimensional subcube and use Lemma~\ref{lemma:subcube}. The 
    subcube will be formed by the vertices, which are added to the infected subset one by one.
    \begin{gather*}
     x_4=0^{a_1}1^k0^{n-a_1-k}, \quad x_5=0^{a_1-1}1^{k+1}0^{a_2-k}10^{n-a_1-a_2-1},\\
     x_6=0^{a_1-1}1^{k+1}0^{n-a_1-k}, \quad x_7=0^{a_1}1^k0^{a_2-k}10^{n-a_1-a_2-1}.
    \end{gather*}
We follow the writing pattern used in the proof of Lemma~\ref{lemma:triangle}. The infection sequence is as follows:
   \begin{gather*}
   (x_1,x_2,x_3)\to x_4,\quad
   (x_2,x_3,x_4)\to x_5,\\
   (x_2,x_4,x_5)\to x_6,\quad
   (x_4,x_5,x_6)\to x_7.
   \end{gather*}
The remaining cases follow the same pattern.
\item $a_1=0$ and $a_3 >0$. The subcube that proves overall percolation is $Q_{k,2}=(x_4,x_5,x_6,x_7)$, where
  \begin{align*}
   &x_4=1^k0^{n-k} , \quad   x_5=1^{k+1}0^{n-k-1},\\ &x_6=1^{k+2}0^{n-k-2}, \quad x_7=1^{k}010^{n-k-2}.
  \end{align*}
We have
  \begin{equation}
  \begin{split}\label{eq:infection-steps1}
    (x_1,x_2,x_3)\to x_4,\quad
    (x_2,x_3,x_4)\to x_5,\\
    (x_3,x_4,x_5)\to x_6,\quad
    (x_4,x_5,x_6)\to x_7.
  \end{split}
  \end{equation}

 \item $a_1 >0$ and $a_3=0$. The proof is the same as in the previous case by symmetry.
\end{enumerate}

\vspace*{.1in}
\paragraph{\bf 2. $a_2 =k$} In this case $a_1\le k$ and $a_3\le k$.
  \begin{enumerate}[label=2.\arabic*]
      \item $0 <a_1 \leq k$ and $ 0<a_3 < k.$ We have $Q_{k,2}=(x_4,x_5,x_6,x_7)$, where
    \begin{align*}
        &x_4=0^{a_1}1^{a_2}0^{n-a_1-a_2}, \quad x_5=0^{a_1-1}1^{a_2+1}0^{n-a_1-a_2},\\
        & x_6=0^{a_1}1^{a_2+1}0^{n-a_1-a_2-1},\quad x_7=0^{a_1-1}1^{a_2+2}0^{n-a_1-a_2-1}
    \end{align*}
and these vertices are infected in 4 steps starting with $x_1,x_2,x_3$ and following the sequence in \eqref{eq:infection-steps1}.

      \item $a_1=0$ and $ 0<a_3 < k.$ We have $Q_{k,2}=(x_2,x_4,x_5,x_6)$, where
    \begin{align*}
        &x_4=1^{k-1}0^{n-k+1}, \quad x_5=1^{k+1}0^{n-k-1},\quad x_6=1^{k-1}010^{n-k-1}
    \end{align*}
and these vertices are infected in 3 steps starting with $x_1,x_2,x_3$ and following the sequence 
\begin{equation}\label{eq:2.2}
\begin{aligned}
    (x_1,x_2,x_3)\to x_4,\quad (x_2,x_3,x_4)\to x_5,\quad (x_2,x_4,x_5)\to x_6.
\end{aligned}
\end{equation}

      \item $ 0<a_1 \leq k$ and $ a_3=0.$ 
      If $ 0<a_1 <k$ the proof is the same as in 2.2 due to symmetry. Therefore we can assume that $a_1=k$.       
      We have $Q_{k,2}=(x_3,x_4,x_5,x_6)$, where
    \begin{align*}
        &x_4=0^{k-1}1^k0^{n-2k+1}, \quad x_5=0^{k-1}1^{k+1}0^{n-k},\\
        & x_6=0^{k}1^{k-1}0^{n-2k+1},
    \end{align*}
and these vertices are infected in 3 steps starting with $x_1,x_2,x_3$ and following the sequence 
\eqref{eq:2.2}.

      \item $ a_1=a_3= k.$ We have $Q_{k,2}=(x_4,x_5,x_6,x_7)$, where
    \begin{align*}
        &x_4=0^{k}1^k0^{n-2k}, \quad x_5=0^{k-1}1^{k+2}0^{n-2k-1},\\
        & x_6=0^{k}1^{k+1}0^{n-2k-1}, \quad x_7=0^{k-1}1^{k+1}0^{n-2k}.
    \end{align*}
These vertices are infected in 4 steps starting with $x_1,x_2,x_3$ and following the sequence 
\eqref{eq:infection-steps1}.
  \end{enumerate}

\paragraph{\bf 3. $a_2 <k$}

  \begin{enumerate}[label=3.\arabic*]
      \item $ a_1+a_3 \leq k.$
    \begin{enumerate}[label=3.1.\arabic*]
      \item $ a_1 >0, a_2 >0$, and $a_3 >0.$ We have $Q_{k,2}=(x_4,x_5,x_6,x_7)$, where
    \begin{align*}
        &x_4=0^{a_1}1^{a_2}0^{n-a_1-a_2}, \quad x_5=0^{a_1-1}1^{a_2+2}0^{n-a_1-a_2-1},\\
        & x_6=0^{a_1}1^{a_2+1}0^{n-a_1-a_2-1},\quad x_7=0^{a_1-1}1^{a_2+1}0^{n-a_1-a_2}.
    \end{align*}
These vertices are infected in 4 steps starting with $x_1,x_2,x_3$ and following the sequence in \eqref{eq:infection-steps1}.
 
      \item $ a_1 >0, a_2 =0$, and $a_3 >0.$ We have $Q_{k,2}=(x_1,x_4,x_5,x_6)$, where
    \begin{align*}
        &x_4=10^{n-1}, \quad x_5=010^{n-2},\quad x_6=110^{n-2}.
    \end{align*}
These vertices are infected in a single step starting with $x_1,x_2,x_3$:
  \begin{equation*} 
 (x_1,x_2,x_3)\to x_4,\quad
    (x_1,x_2,x_3)\to x_5,\quad (x_1,x_2,x_3)\to x_6.
   \end{equation*}

      \item\label{item: +0+} $ a_1 >0, a_2 =0$, and $a_3 >0.$ We have $Q_{k,2}=(x_2,x_4,x_5,x_6)$, where
    \begin{align*}
        &x_4=1^{a_2+1}0^{n-a_2-1}, \quad x_5=1^{a_2+2}0^{n-a_2-2},\quad x_6=1^{a_2}010^{n-a_2-2}
    \end{align*}
and these vertices are infected in 2 steps starting with $x_1,x_2,x_3$ and following the sequence 
  \begin{equation*}
  \begin{split}
    &(x_1,x_2,x_3)\to x_4,\quad
    (x_2,x_3,x_4)\to x_5,\\
    &(x_2,x_3,x_4)\to x_6.\quad 
  \end{split}
  \end{equation*}

  \item $ a_1 >0, a_2 >0$, and $a_3 =0.$ 
Due to symmetry the proof is the same as in Case~\ref{item: +0+}.
      \end{enumerate}

      \item $ a_1+a_3 > k,$ and $k \geq 3$.

      \begin{enumerate}[label=3.2.\arabic*]
    \item\label{item: >k<k<k-1a}
    $a_3 > k$, $a_1 < k$, $a_2< k-1$ and $a_1 \geq a_2+2$. 
Let $x_4=0^{a_1-l_1}1^{a_2+l_2+l_1}0^{n-a_1-a_2-l_2}$.

The vertices $x_1$, $x_2$ and $x_3$ will fall in $N(x_4)$ if
\begin{equation}\label{eq:l1l2}
\left.\begin{aligned}
l_1+l_2+a_2 \leq k,  \\
l_1+a_3-l_2 \leq k,  \\
l_2+a_1-l_1 \leq k.   
\end{aligned}
\right\}
\end{equation}

Let 
    $$
    l_1=\left \lfloor \frac{a_1-a_2}{2} \right \rfloor, \quad l_2=k-\left \lceil \frac{a_2+a_1}{2} \right \rceil.
    $$
It is easy to check that $l_1$ and $l_2$ satisfy (\ref{eq:l1l2}). It is also easy to see that $a_1 \geq l_1$ and $a_3 \geq l_2.$ 
Finally, $a_1\geq a_2+2$ implies that $l_1 \geq 1$ and  $ a_2+2 \leq a_1 <k$ implies $l_2 \geq 1$.

\item $a_3 > k$, $a_1 < k$, $a_2< k-1$ and $a_1 \leq a_2$.
Let $l_1=0.$ and  $l_2=k-a_2.$
It is easy to check that if $a_1\leq a_2$ then $l_1$ and $l_2$ satisfy \eqref{eq:l1l2}. 

\item $a_3 > k$, $a_1 < k$, $a_2< k-1$, $a_1=a_2+1$ and $a_2 \geq \frac{k}{3}$.

Let 
$$l_1=k-\left \lceil \frac{a_2+a_3}{2} \right \rceil$$
 and 
 $$l_2=\left \lfloor \frac{a_3-a_2}{2} \right \rfloor.$$
It is easy to check that $l_1$ and $l_2$ satisfy (\ref{eq:l1l2}). Note that $l_1 \geq 1$ since 
    $$
    a_2+a_3 \leq 2k-2
    $$ 
and that $a_1 \geq l_1$ which requires 
\begin{align}
    2a_1+a_2+a_3 \geq 2k. \label{equation4}
\end{align}
In order for (\ref{equation4}) to be satisfied it is sufficient that $a_2 \geq \frac{k}{3}$. 

    \item\label{item: >k<k<k-1b} $a_3 > k$, $a_1 < k$, $a_2< k-1$, $a_1=a_2+1$ and $a_2 < \frac{k}{3}$.
Let $l_1=0$ and $l_2=k-a_1$. It is easy to check that $l_1$ and $l_2$ satisfy (\ref{eq:l1l2}). 
It is also easy to see that $l_2 \geq 2$ as long as $k \geq 3$, and this implies that $a_1 \leq k-2$. 

Now for Cases \ref{item: >k<k<k-1a}--\ref{item: >k<k<k-1b}, we need to construct vertices $x_5,x_6$ and $x_7$ such 
that each of them will be infected by the percolation process, and together with $x_4$ they form a subcube $Q_{2,k}$. 
Given $l_1 \geq 1 $ and $l_2 \geq 1$, we define these vertices in the form
   \begin{align*}
            & x_5=0^{a_1-l_1+1}1^{a_2+l_2+l_1-1}0^{n-a_1-a_2-l_2},\\
&x_6=0^{a_1-l_1}1^{a_1+l_2+l_1-1}0^{n-a_1-a_2-l_2+1},\\
              & x_7=0^{a_1-l_1+1}1^{a_2+l_2+l_1-2}0^{n-a_1-a_2+1}
   \end{align*}
and note that they are infected starting with $x_1,x_2,x_3$ and following the sequence 
  \begin{equation*} 
    (x_1,x_3,x_4)\to x_5,\quad (x_1,x_2,x_4)\to x_6,\quad (x_4,x_5,x_6)\to x_7.
  \end{equation*}
Now given $l_1=0$ and $l_2 \geq 2$, we have 
\begin{align*}
&  x_5=0^{a_1}1^{l_2+a_2-1}0^{n-a_1-a_2-l_2+1},
\quad x_6=0^{a_1}1^{l_2+a_2-2}10^{n-a_1-a_2-l_2+1},\\
& x_7=0^{a_1}1^{l_2+a_2-2}0^{n-a_1-a_2+2},
\end{align*}
where these vertices are infected starting with $x_1,x_2,x_3$ and following the sequence 
  \begin{equation*} 
    (x_1,x_2,x_4)\to x_5,\quad
    (x_1,x_2,x_4)\to x_6,\quad
    (x_4,x_5,x_6)\to x_7.
  \end{equation*}

\item  $a_3 > k$, $a_1 < k$ and $a_2= k-1$

Since $a_2=k-1$, then $a_3=k+1$. 
We have $Q_{k,2}=(x_4,x_5,x_6,x_7)$, where 
\begin{align*}
& x_4=0^{a_1}1^{k}0^{n-a_1-k}, \quad x_5=0^{a_1+1}1^{k}0^{n-a_1-k-1},\\
& x_6=0^{a_1-1}1^{k+1}0^{n-a_1-k}, \quad x_7=0^{a_1+1}1^{k-1}0^{n-a_1-k},
\end{align*}
and these vertices are infected starting with $x_1,x_2,x_3$ and following the sequence 
\begin{equation*} 
  \begin{split}
    (x_1,x_2,x_3)\to x_4,\quad
    (x_1,x_3,x_4)\to x_5,\\
    (x_3,x_4,x_5)\to x_6,\quad 
    (x_4,x_5,x_6)\to x_7. 
  \end{split}
  \end{equation*}
    \item $a_3 \leq k$, $a_1 \leq k$ and $a_2 >0$.
    We have $Q_{k,2}=(x_4,x_5,x_6,x_7)$, where 
\begin{align*}
& x_4=0^{a_1}1^{a_2}0^{n-a_1-a_2}, \quad x_5=0^{a_1}1^{a_2+1}0^{n-a_1-a_2-1},\\
&  x_6=0^{a_1-1}1^{a_2+1}0^{n-a_1-a_2}, \quad x_7=0^{a_1-1}1^{a_2+2}0^{n-a_1-a_2-1}.
\end{align*}
These vertices will be infected starting with $x_1,x_2,x_3$ and following the sequence 
\begin{equation*} 
  \begin{split}
    (x_1,x_2,x_3)\to x_4,\quad
    (x_1,x_3,x_4)\to x_5,\\
    (x_2,x_4,x_5)\to x_6,\quad 
    (x_4,x_5,x_6)\to x_7. \\
  \end{split}
  \end{equation*}

    \item  $a_3 \leq k$, $a_1 \leq k$ and $a_2 =0$.
    We have $Q_{k,2}=(x_1,x_4,x_5,x_6)$, where the vertices
\begin{align*}
&  x_4=0^{a_1-1}110^{n-a_1-1}, \quad x_5=0^{a_1}10^{n-a_1-1},\\
&   x_6=0^{a_1+1}10^{n-a_1-2}
\end{align*}
will be infected starting with $x_1,x_2,x_3$ and following the sequence 
\begin{equation*} 
  \begin{split}
    &(x_1,x_2,x_3)\to x_4,\quad
    (x_1,x_3,x_4)\to x_5,\\
    &(x_1,x_4,x_5)\to x_6.
  \end{split}
  \end{equation*}
\end{enumerate} 
\end{enumerate}
This concludes the proof.
\end{proof}

This lemma enables us to find the main term of the critical probability of percolation. 
\begin{lemma}[Critical probability for the 3-neighbor process] \label{thm: k>=2}
We have 
     $$ \frac{1}{\omega(n)}2^{-\frac{n}{3}} n^{-k} \leq p_c(Q_{n,k},3) \leq  \omega(n)2^{-\frac{n}{3}} n^{-k}.
     $$
\end{lemma}

\begin{proof}
By Lemma~\ref{lemmaV2}, percolation occurs if and only if the initially infected set $\cA_0$ contains a triangle with all sides 
at most $2k$. Let $X$ be the number of such triangles in $\cA_0$. 
Let $x_1=0^n$, $x_2=1^{a_1+a_2}0^{n-a_1-a_2}$, and $x_3=0^{a_1}1^{a_2+a_3}0^{n-a_1-a_2-a_3}$, where all the pairs $(a_i,a_j)$ 
satisfy $a_i+a_j\le 2k$. Then
  \begin{align*}
 \ee [X]  & = \frac{1}{6}p^32^n\sum_{\begin{substack}{a_1,a_2,a_3\\a_i+a_j\le 2k}\end{substack}}
 \binom{n}{a_1+a_2}\binom{a_1+a_2}{a_2}\binom{n-a_1-a_2}{a_3}.
 \end{align*}
Since $k$ is a constant, for large $n$ the largest term in the sum is given by 
   $
   \binom{n}{2k}\binom{2k}{k}\binom{n-2k}{k}.
   $

Next let us estimate the variance of $X$. Let $T_1=\{x_1,x_2,x_3\}$ and $T_2=\{x_1',x_2',x_3'\}$ be two triangles with 
all the sides at most $2k$. For a triple $T$ define the indicator random variable $X_T={\mathbbm 1}_{\{T\in \cA_0\}}$.
We have 
\begin{align*}
\text{Cov}(X_{T_1},X_{T_2})& =\ee [X_{T_1}X_{T_2}]-\ee [X_{T_1}] \ee  [X_{T_2}]\\
&=p^{|T_1 \cup T_2|}-p^{|T_1|+|T_2|}\\
&=
\begin{cases}
  0& \text{if $|T_1 \cap T_2|=0$} \\
  p^5-p^6 & \text{if $|T_1 \cap T_2|=1$}\\
  p^4-p^6 & \text{if $|T_1 \cap T_2|=2$}\\
  p^3-p^6 & \text{if $|T_1 \cap T_2|=3$}.
\end{cases}
\end{align*}
Thus, we obtain 
\begin{align*}
\Var(X) &=\sum_{T_1,T_2} \text{Cov}(X_{T_1},X_{T_2})\\
& =\sum_{|T_1 \cap T_2|=1} \text{Cov}(X_{T_1},X_{T_2})+\sum_{|T_1 \cap T_2|=2} \text{Cov}(X_{T_1},X_{T_2})+\sum_{|T_1 \cap T_2|=3} \text{Cov}(X_{T_1},X_{T_2})\\
&\leq  C(k)2^n\binom{n}{2k} \binom{n}{k} \Big[(p^5-p^6)  \binom{n}{2k}\binom{n}{k} 
 +(p^4-p^6)\binom{n}{2k}\\
&\quad  +(p^3-p^6) \Big].
\end{align*}
where $C(k)$ is a constant that depends only on $k$. For the first term above,  the number of good triples $T_1$ is at most $C(k)2^n \binom{n}{2k}\binom{n}{k}$ and $T_2$ is at most $C(k)\binom{n}{2k}\binom{n}{k}$ since $|T_1 \cap T_2|=1$. Other terms can be derived in a similar way.

In order to fit the assumption that $\Var(X)=o((\ee X)^2)$, it suffices to take 
$$p^32^n\binom{n}{2k}\binom{n}{k} \gg \left(p^52^n\binom{n}{2k}^2\binom{n}{k}^2\right)^{1/2}$$

$$p^32^n\binom{n}{2k}\binom{n}{k} \gg \left(p^42^n\binom{n}{2k}^2\binom{n}{k}\right)^{1/2}$$
and 
$$p^32^n\binom{n}{2k}\binom{n}{k} \gg \left(p^32^n\binom{n}{2k}\binom{n}{k}\right)^{1/2}.$$

It is easy to check that if $p \gg 2^{\frac{-n}{3}}n^{-k}$, then $\Var(X)=o((\ee X)^2)$ and $\ee X \to \infty$, and the lower bound
on $p_c$ follows by Lemma~\ref{lemma15}. On the other hand, if $p \ll 2^{\frac{-n}{3}}n^{-k}$, then by Markov's inequality, 
$\pp(X \geq 1)\leq \ee [X]=o(1)$.
\end{proof}

If  $p=c2^{-\frac{n}{3}} n^{-k}$, then the probability of full infection converges to a constant. 
\begin{lemma}
Let $c >0$ be a constant and $p=c 2^{-\frac{n}{3}} n^{-k}$. As $n \rightarrow \infty$, we have 
\begin{align*}
 \mathbb{P}(\cA_0 \; \text{percolates}) \rightarrow 1- \exp \left(-\frac{c^3}{6\binom{2k}{k}(2k)! k!}\right). 
\end{align*}
\end{lemma}
\begin{proof}
Define  
\begin{equation*}
  \mathbb{I}_{i}=  
  \begin{cases}
   1 & \text{if $i$th triple $(x,y,z)$ with $d_{xy}\leq 2k$, $d_{yz}\leq 2k$, and $d_{xz}\leq 2k$ is initially infected} \\
  0 & \text{otherwise}.
\end{cases}.
\end{equation*}
Define $N=\frac{1}{6}2^n\binom{n}{2k}\binom{2k}{k}\binom{n-2k}{k}(1+o(1))$. Then we have 
    \begin{align}
    \mathbb{P}(\cA_0 \; \text{does not percolates})
    & = \mathbb{P}(\sum_{i=1}^{N}\mathbb{I}_{i}=0)\\
    & = \mathbb{P}(\cap_i^{N} \{\mathbb{I}_{i}=0\})\\
    & \geq \prod_{i=1}^{N}\mathbb{P}(\{\mathbb{I}_i=0\}). \label{eq2: use FKG} \\
    & =(1-p^3)^N\\
    & \rightarrow \exp \left(-\frac{c^3}{6\binom{2k}{k}(2k)! k!}\right)
    \end{align}
Note that the inequality~\ref{eq2: use FKG} is true due to Lemma~\ref{lemma: FKG}. 

Now let us prove the upper bound by using the inequality in Lemma~\ref{lemma: Janson}.

Let $R$ be the set of initially infected vertices and $\{A_i\}_{i=1}^N$ be a collection of the sets of pair $(x,y)$ with $d_{xy}\leq 2k$. Let $B_i$ be the event that $A_i \subset R$. 

From the proof of Lemma ~\ref{thm: k>=2}, we have  $$\mu:=\mathbb{E}(\sum_{i=1}^N \mathbb{I}_{B_i})=Np^3$$ and  

$$\Delta:=\sum_{i \sim j}\mathbb{P}(B_i \cap B_j) \leq 2^n\binom{n}{2k} \binom{n}{k} \Big[p^5  \binom{n}{2k}\binom{n}{k} +p^4\binom{n}{2k} 
\Big] \leq o(1)$$
where $i \sim j$ if $A_i \cap A_j \neq \emptyset$ for $i \neq j$. 

By Lemma~\ref{lemma: Janson}, we have the desired result. 
\end{proof}
Similar to the case where $r=2$, we can characterize the distribution of the number of unordered triples $(x,y,z)$ satisfying $d_{xy}\leq 2k$, $d_{xz} \leq 2k$ and $d_{yz} \leq 2k$. 
\begin{lemma}
 If $p=c2^{-\frac{n}{3}} n^{-k}$, then the number of unordered triples $(x,y,z)$ such that $d_{xy}\leq 2k$, $d_{xz} \leq 2k$ and $d_{yz} \leq 2k$ is asymptotically Poisson distributed with mean $\frac{c^3}{6\binom{2k}{k}(2k)! k!}$.   
\end{lemma}

\begin{proof}
 We follow the ideas from Chapter 5 of \cite{Random_graphs_Frieze2} by using Theorem~\ref{thm:random graphs} and Remark~\ref{rmk:random graphs}.  Let $H_1,H_2,\cdots, H_t$ be all the distinct unordered triples $(x,y,z)$ with $d_{xy} \leq 2k$, $d_{xz} \leq 2k$, and $d_{yz} \leq 2k$ on the graph $Q_{n,k}$. For $i=1,2,\cdots,t$, let 
\begin{align*}
 \mathbb{I}_{H_i} = 
    \begin{cases}
  1 & \text{if $H_i \subset \cA_0$} \\
   0 & \text{otherwise.} \\
\end{cases}
\end{align*}
Let $X_H=\sum_{i=1}^t\mathbb{I}_{H_i}$ and the $\ell$th factorial moment of $X_{H}$, $\ell=1,2,\cdots$, 
$$\mathbb{E}(X_H)_{\ell}=\mathbb{E}[X_H(X_H-1)\cdots(X_H-\ell+1)],$$
can be written as 
\begin{align*}
\mathbb{E}(X_H)_\ell & =\sum_{i_1,i_2,\cdots,i_{\ell}}\mathbb{P}(\mathbb{I}_{H_{i_1}}=1,\mathbb{I}_{H_{i_2}}=1,\cdots,\mathbb{I}_{H_{i_{\ell}}}=1)\\
& =D_{\ell}+\overline{D}_{\ell},
\end{align*}
where the summation is taken over all $\ell$ element sequences of distinct indices $i_j$ from $\{1,2,\cdots,t\}$, and $D_{\ell}$ and $\overline{D}_{\ell}$ denote the partial sums taken over all ordered $\ell$ tuples of unordered triples $(x,y,z)$ with $d_{xy}\leq 2k$, $d_{xz} \leq 2k$, and $d_{yz} \leq 2k$, which are, respectively, pairwise disjoint ($D_\ell$) and not all pairwise disjoint $\overline{D}_\ell$. 

Now we have 
\begin{align*}
    D_{\ell}& =\sum_{i_1,i_2,\cdots,i_{\ell}} \prod_{j=1}^{\ell}\mathbb{P}(\mathbb{I}_{H_{i_j}}=1)\\
    & =(1+o(1))\left(\frac{1}{6}2^{n}\binom{n}{2k}\binom{2k}{k}\binom{n-2k}{k} \right) \left(\frac{1}{6}\left(2^{n}-3 \right)\binom{n}{2k}\binom{2k}{k}\binom{n-2k}{k}\right)\cdots \\
    &\left(\frac{1}{6}(2^{n}-3(\ell-1))\binom{n}{2k}\binom{2k}{k}\binom{n-2k}{k} \right)(p^3)^{\ell}\\
    & =(1+o(1))\left(\frac{1}{6}2^{n}\binom{n}{2k}\binom{2k}{k}\binom{n-2k}{k} \right) ^{\ell}p^{3\ell}
\end{align*}
and 
\begin{align*}
\left(\mathbb{E}[X_H]\right)^\ell=(1+o(1))\left(\frac{1}{6}2^{n}\binom{n}{2k}\binom{2k}{k}\binom{n-2k}{k} \right) ^{\ell}p^{3\ell}
\end{align*}
Note that $\left(\frac{1}{6}2^{n}\binom{n}{2k}\binom{2k}{k}\binom{n-2k}{k} \right)^{\ell}p^{3\ell}=(1+o(1))\left(\frac{c^3}{6\binom{2k}{k}(2k)! k!}\right)^{\ell}$. If we can show that 
$$\overline{D}_k=o(1),$$
then the proof is finished. 

Let us consider the family $\mathcal{G}_\ell(V,E)$ be the collection of all simple and undirected graphs with at least one edge, where the vertex set $V$ is $\ell$ distinct unordered triples $(x,y,z)$ such that $d_{xy}\leq 2k$, $d_{xz} \leq 2k$ and $d_{yz} \leq 2k$, and the edge set $E=\{(x_1,y_1,z_1)(x_2,y_2,z_2): \{x_1,y_1,z_1\} \cap \{x_2,y_2,z_2\} \neq \emptyset \}$. Note that there is a one-to-one correspondence between the set $\mathcal{G}_\ell(V,E)$ and the set of ordered $\ell$ tuples of unordered triples $(x,y,z)$ with $d_{xy}\leq 2k$, $d_{xz} \leq 2k$ and $d_{yz} \leq 2k$, which are not all pairwise disjoint. 

Now let us partition the graphs in this collection according to the number of components in the graph. Let $U_{\ell,i}$ denote a the collection of graphs with $i$ components and we have 
$$\mathcal{G}_{\ell}=\bigcup_{i=1}^{\ell-1} U_{\ell,i}. $$
Now let us consider a particular graph $G$ on the collection $U_{\ell,i}$. Since $G$ has $i$ components, it is not difficult to see that the vertex set of $G$ contains at least $3i+1$ distinct elements from $V(Q_{n,k})$. Moreover, it is easy to see that the number of of such ordered $\ell$ tuples of unordered triples $(x,y,z)$ with $d_{xy}\leq 2k$, $d_{xz} \leq 2k$ and $d_{yz} \leq 2k$ is at most $2^{ni}\binom{n}{2k}^{3l}$ Therefore, we have 
$$\overline{D}_\ell\leq \sum_{i=1}^{\ell-1}\sum_{U_{\ell,i} \in \mathcal{G}_{\ell}}p^{3i+1}2^{ni}n^{6kl} \leq \ell 2^{\binom{\ell}{2}}p^{3i+1}2^{ni}n^{6kl} =o(1),$$
which is the desired result. \qedhere
\end{proof}

Lemma~\ref{lemmaV2} gives a sufficient condition for bounding the critical probability, and one would want to extend it to all $r$. 
Unfortunately, this extension fails, as will be shown next.
\begin{lemma}
Let $r\geq 4$ and $rk\le n$. There exists a subset $\cA_0=\{x_1,x_2,\dots ,x_r\}$ of $r$ vertices with all pairwise distances $\le 2k$ such that the $r$-neighbor percolation process that starts with $\cA_0$ fails to infect the entire graph $Q_{n,k}$.
\end{lemma}
\begin{proof} We will show that the set $\cA_0=\{x_1,x_2,\dots ,x_{r-1},x_r=0^n\}$ with $x_i=0^{(i-1)k}1^k0^{n-ik},i=1,\dots,r-1$ and $d(x_i,x_j)\le 2k$ for all $1\le i,j\le r$ satisfies the claim of the lemma. We claim that 
    $$
    \bigcap_{i=1}^r N_{x_i}(Q_{n,k})=\emptyset.
    $$

Suppose the contrary and let $y \in \cap_{i=1}^r N(x_i)$ be a vertex common to all the neighborhoods. Suppose that
the Hamming weight of the subvector $(y_{(j-1)k+1},\dots,y_{jk})$ is $a_j$, where $0\le a_j\le k$ and $j=1,\dots,r-1$.
We have $$0<\sum_{i=1}^{r-1}a_i \leq k$$
and for $j\in\{1,\dots,r-1\},$
   $$
   \sum_{\substack{i=1 \\ i\neq j}}^{r-1}a_i+k-a_j \leq k.
   $$
Then we have 
$$(r-3)\left(\sum_{i=1}^{r-1}a_i \right) \leq 0,$$
which is a contradiction since it forces $\sum_{i=1}^{r-1} a_i \leq 0$.
\end{proof}

\section{Minimum percolating sets} \label{sec:MPS}
In this section, we will prove Theorem~\ref{thm:max-set} and Theorem~\ref{thm: min-set}. Let us first prove Theorem~\ref{thm:max-set}. 

\begin{proof}
Let us prove \eqref{eq:8-n}. Clearly, $m(r)\ge r$ because otherwise the initial set $\cA_0$ will never change. Repeating 
the proof of Lemma~\ref{lemma:subcube}, we claim that if the initially infected set includes the subcube 
$Q_{r-1,k}^0$, then the $r$-neighbor process will infect the entire graph. Again, in the first step,  infection
spreads to every vertex $v$ in the subcube $\{\ast^{r-1}10^{n-r}\}$ since $|N_{v}\cap Q_{r-1,k}^0|=r$, and induction completes the argument.

Thus, if $Q_{r-1,k}^0\subset \cA_0$, then percolation occurs. Denote by $V$ the set of vertices of $Q_{r-1,k}^0$ and let us write
$V=\cup_{j=0}^{r-1} V_j$, where $V_j$ is the set of vertices of Hamming weight $j=0,\dots,r-1$ in $Q_{r-1,k}^0$.
We claim that if  $(V_{m_r-1}\cup V_{m_r})\subset \cA_0$, then the process propagates to the entire subcube $Q_{r-1,k}^0$. 
Indeed, every vertex of weight $m_r-i$ is adjacent to $\binom{r-1-m_r+i}{2}$ vertices of weight $m_r-i+2$. By a direct check,
if $r\ge 10$ then
\begin{align}
   \binom{r-1-m_r+i}{2}\geq r. \label{equation MPS1}
\end{align}

Moreover, every vertex of weight $m_r+i$ is adjacent to $\binom{m_r+i}{2}$ vertices of weight $m_r+i-2$. We have 
\begin{align}
\binom{m_r+i}{2} \geq r. \label{equation MPS2}
\end{align}

Once every vertex in $V_{m_{r}-1} \cup V_{m_{r}}$ is initially infected, then in the next step every vertex on $V_{m_{r}-2}$ and $V_{m_{r}+1}$ will be infected. After that the infection will go on due to the monotonicity in $i$ on the left side of  (\ref{equation MPS1}) and (\ref{equation MPS2}). Therefore, we have $ m(r) \leq \binom{r-1}{m_r}+\binom{r-1}{m_r-1}$,  which shows \eqref{eq:8-n}. 

Now let us prove (\ref{eq:>n}), which is very similar to that for (\ref{eq:8-n}). First let us write $V(Q_{n,2})=\cup_{j=0}^nV_j(Q_{n,2})$, where $V_j(Q_{n,2})$ is the set of vertices of Hamming weight $j=0,\dots,n$ in $V(Q_{n,2})$. We claim that if $(V_{m_r} \cup V_{m_r+1}) \subset \cA_0$, the process infects the entire graph $Q_{n,2}$. Indeed every vertex of weight $m_r-i$ is adjacent to $\binom{n-(m_r-i)}{2}$ vertices of weight $m_r-i+2$. As long as $r \leq \frac{n^2}{8}$, 
\begin{align}
\binom{n-m_r+i}{2} \geq r, \label{eq: n-m_r+i}
\end{align}

Moreover, every vertex of weight $m_r+i$ is adjacent to $\binom{m_r+i}{2}$ vertices of weight $m_r+i-2$. We have 
\begin{align}
\binom{m_r+i}{2} \geq r. \label{eq: m_r+i}
\end{align}

Once every vertex in $V_{m_{r}} \cup V_{m_{r}+1}$ is initially infected, then in the next step every vertex on $V_{m_{r}-1}$ and $V_{m_{r}+2}$ will be infected. After that the infection will go on due to the monotonicity in $i$ on the left side of (\ref{eq: n-m_r+i}) and (\ref{eq: m_r+i}). Therefore, we have $r \leq m(r) \leq \binom{n}{m_r}+\binom{n}{m_r+1}$,  which shows \eqref{eq:8-n}. 

Now let us prove (\ref{eq: r=cn}). Again, if $Q_{r-1,k}^0\subset \cA_0$, then percolation occurs. Denote by $V$ the set of vertices of $Q_{r-1,k}^0$ and let us write
$V=\cup_{j=0}^{r-1} V_j$, where $V_j$ is the set of vertices of Hamming weight $j=0,\dots,r-1$ in $Q_{r-1,k}^0$. We claim that if $\cup_{j=0}^{k-1}V_{m_{k,r}+j} \subset \cA_0$, then every vertex in $Q_{r_1,k}^0$ will be infected by the process. Indeed, every vertex of weight $m_{k,r}-i$ is adjacent to $\binom{r-1-(m_{k,r}-i)}{k}$ vertices of weight $m_{k,r}-i+k$. 
If $r \geq 2^{\frac{k^2}{k-1}}$, by using the $\binom{n}{m} \geq \left(\frac{n}{m} \right)^m$, we have 
\begin{align}
\binom{r-1-(m_{r,k}-i)}{k} \geq r. \label{eq: r-1 k,r}
\end{align}

Moreover, every vertex of weight $m_{k,r}+i$ is adjacent to $\binom{m_{k,r}+i}{k}$. We have for $i \geq k$
\begin{align}
\binom{m_{k,r}+i}{k} \geq r.   \label{eq: k r}
\end{align}

Once every vertex in $\cup_{j=0}^{k-1} V_{m_{k,r}+j}$ is initially infected, then in the next step every vertex on $V_{m_{k,r}-1}$ and $V_{m_{k,r}+k}$ will be infected. After that the infection will go on due to the monotonicity in $i$ on the left side of (\ref{eq: r-1 k,r}) and (\ref{eq: k r}). Therefore, we have $ m(r) \leq \sum_{j=0}^{k-1} \binom{r-1}{m_{k,r}+j}$,  which shows \eqref{eq: r=cn}. Similarly, \eqref{eq: r less} can be proved.  

Before proceeding, we will need the following lemma. 
\begin{lemma} \label{lemma3}
 Suppose that $k\ge 2, n\ge r\ge 2,$ and $2 \leq r \leq 2k$. Then $$ m(r)=r.$$ 
\end{lemma} 
\begin{proof}
By Lemma~\ref{lemma:subcube}, it suffices to exhibit a set $\cA_0$ of size $r$ that implies percolation. First 
suppose that $r-1 \leq k$. Let $\cA_0=\{x_1,\dots ,x_r\}$, where $x_1=0^n$, $x_2=1^{r-1}0^{n-r+1},$ and 
$x_3=1^{r-2}0^{n-r+2},$ $x_4=1^{r-3}010^{n-r+1}$, etc., $x_r=101^{r-3}0^{n-r+1}$.

Since $r-1 \leq k$ every binary sequences in $\{0,1\}^{r-1}$ are within Hamming distance $k$ from $x_1,\dots ,x_r$. Therefore, every vertex on this $r-1$ dimensional cube will be infected. 

     Now let us consider the case where $r-1 \geq k$. Again let us initially infect $x_1,\dots ,x_r$. It is easy to see that every binary sequence of length $r-1$ and of weight $k$ will be infected by the process since those sequences have $x_1,\dots ,x_r$ as their neighbors. Then every binary sequence of length $r-1$ and of weight $k+1$ will be infected by the process and every binary sequence of length $r-1$ and of weight up to $r-2$ will be infected by the process. Therefore, every vertex in this $r-1$ dimensional cube is infected by this process. 
\end{proof}

Returning to the proof of Theorem~\ref{thm:max-set}, we will now address \eqref{eq:2-6}, which has been mostly taken care
of by Lemma~\ref{lemma3}. What remains is to show that 
$m(5)=5,$
 and 
$m(6)=6.$
Let $x_i=1^{i-1}0^{n-i+1}, i=1,\dots,5$. We claim that if $\cA_0$ contains these 5 vectors, the process will percolate to the entire graph.
Indeed, observe that the first steps of the process propagate to the subcube $Q_{4,k}^0$, and after that Lemma~\ref{lemma:subcube}
implies the full claim $m(5)=5$.

Moving to the case $r=6$, consider the set 
   $$\cA_0=\{0^{n},01^20^{n-3},0^21^30^{n-5}, 1^40^{n-4}, 01^40^{n-5}, 0^21^40^{n-6}\}.$$
The proof is completed by observing that the first steps infect the subcube $Q_{6,k}^0$ and using Lemma~\ref{lemma:subcube} to prove
the full claim.

Next we address \eqref{eq:>15}. Let $V(Q_{n,k})=V_0\cup V_1\cup\dots\cup V_n$ be the set of vertices of the cube with $V_i$
denoting the subset of all vertices of Hamming weight $i$, for all $i$. 
Without loss of generality assume that $0^n \in \cA_1$, then the $r$ initially infected vertices in $\cA_0$ should be of weights $1$ to $2$. 
Let $v\in V_3$, then $|N_v\cap V_2|=3$ and $|N_v\cap V_1|=3.$ Similarly, if $w\in V_4$, then $|N_w\cap V_2|=6$. Thus even if
$\cup_{i=0}^2V_i\subset \cA_0$, percolation does not occur.

It remains to prove \eqref{eq:>3}. As before, let $V_i, i=0,\dots,n$ be the subset of vertices of weight $i$. We claim that for
the $r$-process to percolate it suffices that $|V_1\cap\cA_0|=r$. Indeed, since the vertices in $V_0$ and $V_1$ form a clique, 
the process clearly spreads to every vertex in $V_0\cup V_1$. Next, since a vertex in $V_2$ is connected to all of $V_0\cup V_1$, 
so the process further spreads to $V_2$. Continuing in this way, every vertex in $V_i$ 
has at least $\binom{i}{i-1}+\binom{i}{i-2}(n-i)$ neighbors in $V_{i-1}$ since $k \geq 3$. This suffices to claim percolation. \end{proof}

In order to prove Theorem~\ref{thm: min-set}, we will need the following lemma. 

\begin{lemma} \label{lemma: lower bound}
 Suppose that $k\geq 2$ and $n$ is sufficiently large. If $\sum_{i=1}^k\binom{n'}{i} \leq cr^{\ell}$ and $\sum_{i=1}^{k-1}\sum_{j=1}^{k-i} \binom{n'}{j}\binom{n-n'}{i} \le r-cr^{\ell}$ with $0 < \ell \le 1$ and $0 <c <1$, then 
 $$m(r) \ge \delta2^{n'},$$
 where $\delta >0$. 
\end{lemma} 
\begin{proof}
 It is not difficult to see that percolation fails to occur from some initially infected set if and only if there exists a set $B$ of healthy vertices such that every vertex in $B$ has fewer than $r$ neighbors outside of $B$. Therefore our task is to show the existence of a such set $B$ with $|\cA_0|=\delta 2^{n'}$ regardless of the locations of the vertices in $\cA_0$.

    Consider a partition of $V(Q_{n,k})$ into $2^{n-n'}$ parts, i.e.,
    $$V(Q_{k,n})=\bigcup_{a \in \{0,1\}^{n-n'}}[\ast]^{n'}a,$$
    where $[\ast]^{n'}a$ denotes $\underbrace{\ast \ast \cdots \ast}_{n'} a$. 
    Let us denote the induced subgraphs $\{G_a\}_{a \in \{0,1\}^{n-n'}}$ with $V(G_{a})=[\ast]^{n'}a$  and $E(G_a)=\{xy: 1 \le d_{xy} \le k \; \text{and}\;  x, y \in V(G_a)\}$.
    
    Note that each induced subgraph $G_a$ is an $n^{'}$-dimensional subcube and each vertex of $G_a$ has $cr^{\ell}=\sum_{i=1}^k\binom{n'}{i}$ neighbors in $G_a$, where $ 0<c<1 $. We will argue that if $|\cA_0|= \delta 2^{n'}$ with $\delta >0$, then there exists a set $B$ of healthy vertices such that every vertex in that set has less than $r$ neighbors outside that set and thus percolation does not happen. 

    Since $|\cA_0|=\delta 2^{n'}$, there is a set $B_a$ of healthy vertices of size $(1-\delta)2^{n'}$ in $G_a$ for each $a \in \{0,1\}^{n-n'}$. Additionally, the vertices in the sets $(B_a)_{a \in {\{0,1\}^{n-n'}}}$ satisfy the following condition: if $x_1x_2\cdots x_{n'}a$ is in $B_a$ then $x_1x_2\cdots x_{n'}a'$ is in $B_{a'}$ for all $a,a' \in \{0,1\}^{n-n'}$. Therefore, $B=\cup_{a \in \{0,1\}^{n-n'}} B_a$. 

    Now let us count the number of possibly infected neighbors not in $V(G_a)$ for a vertex in $B_a$. W.l.o.g, we can assume that $a=[0]^{n-n'}$. Let us denote the number of possibly infected neighbors not in $V(G_a)$ for a vertex $v$ in $B_{[0]^{n-n'}}$ by $N(v)$. It is not difficult to see that 
    \begin{align}
    N(v) \le \sum_{i=1}^{k-1}\sum_{j=1}^{k-i} \binom{n'}{j}\binom{n-n'}{i}. \label{eq: N(v)}
    \end{align}
    Therefore, as long as $cr^{\ell}+N(v) \leq r$, we can conclude $m(r) \ge \delta 2^{n'}$ where $\delta>0$. \qedhere  
\end{proof}

Now we are ready to prove Theorem~\ref{thm: min-set} and its proof is a direction calculation from Lemma~\ref{lemma: lower bound}. 

\begin{proof}
Let us show \eqref{eq:k=2}. For $k=2$, we have $\binom{n'}{2}+n' \leq cr^l$ and thus $n' \leq \sqrt{2cr^{\ell}}$. From \eqref{eq: N(v)}, we have $(n-n')n' \leq (1-c)r$ and thus $n \leq \sqrt{2cr^{\ell}}+\frac{(1-c)r^{1-\ell/2}}{\sqrt{2c}} \leq \frac{2(1-c)r^{1-\ell/2}}{\sqrt{2c}}$  since $1-\frac{\ell}{2}\geq \frac{\ell}{2}$ and $n$ is sufficiently large. We have the desired result.

\eqref{eq: k=2,l=1} can be proved very similarly. 

Now let us prove \eqref{eq: k>=3,l<1}. For $k \geq 3$, we have $\sum_{i=1}^k\binom{n'}{i} \leq cr^{\ell}$ and thus $n' \leq r^{\ell/k} \frac{k}{e} \left(\frac{c}{2}\right)^{1/k} $. From \eqref{eq: N(v)}, we have 
$\sum_{i=1}^{k-1}\sum_{j=1}^{k-i} \binom{n'}{j}\binom{n-n'}{i} \leq (1-c)r$. It suffices to have $2\binom{n-n'}{k-1}n' \leq (1-c)r$ since $n$ is sufficiently large. 

\eqref{eq: k>=3} can be proved similarly. \qedhere

\end{proof}

\section{Discussion and open problems}
We conclude with several problems left open in our research.

\begin{enumerate}
\item What is $m(Q_{n,k},r)$ when $k\geq 2$ and $r=r(n)$ is large ? 

\noindent We conjecture that if $r=r(n)$ is a relatively fast growing function of $n$, then $m(Q_{n,k},r) \geq \delta 2^{n'}$ where $n'$ satisfies that $\sum_{i=1}^k\binom{n'}{i}\leq cr$ with $0<c <1$ and $\delta >0$. Note that if this holds, then we obtain an exponentially tight bound for $m(Q_{n,k},r)$ over a wide range values of $r$, whereas Theorem ~\ref{thm: min-set} gives such a bound only when $r=\theta(n^k)$. 

In order to explain our conjecture it is necessary to introduce the notion of the weak saturation introduced by Bollob{\'a}s \cite{Bollobas}. Given fixed graphs
$G$ and $H$, a spanning subgraph $F$ of $G$ is weakly ($G$, $H$)-saturated if the edges
of $E(G) \backslash E(F)$ can be added to $F$, one edge at a time, in such a way that each edge
creates a new copy of $H$ when it is added. The weak saturation number of $H$ in $G$ is defined by 
    $$
    \text{wsat}(G,H):= \text{min} \{|E(F)|: F \: \text{is weakly} \: (G,H)\text{-saturated}\}.
    $$
Our conjecture is based on the following bound:
\begin{align}
 m(G,r) \geq \frac{\text{wsat}(G,S_{r+1})}{r},  \label{bound} 
\end{align}
where $S_{r+1}$ denotes a star with $r+1$ rays. 

In \cite{MORRISON201861}, Morrison and Noel used bound \eqref{bound} to determine tight asymptotics for $m(G,r)$, where $G$ is the $n$-dimensional Boolean cube and $r$ is a fixed constant. In their proof they relied on a linear-algebraic method developed in \cite{Balough2012} to obtain an asymptotically tight lower bound for wsat($G$,$S_{r+1}$). We conjecture that 
    $\text{wsat}(Q_{n,k},S_{r+1}) \geq \delta r 2^{n'}.$
To support this conjecture, consider a partitioning of $Q_{n,k}$ into $2^{n-n'}$ isomorphic parts so that each induced subgraph is an $n'$-dimensional subcube. Each vertex of an $n'$-dimensional subcube has $cr=\sum_{i=1}^k\binom{n'}{i}$ neighbors, and the total number of edges in the subcube is $2^{n'-1}cr$. We believe that the lower bound $\text{wsat}(Q_{n,k},S_{r+1}) \geq \delta 2^{n'}$ can be proved relying on a version of the linear algebraic method.

\item What is $p_c(G,r)$ for a fixed $k$ and $r$?

Although Lemma \ref{lemmaV2} does not hold when $r \geq 4$, it might still be possible to derive asymptotics of the critical probability using alternative characterizations of percolating sets. 

\end{enumerate}

\section*{Acknowledgment}
The author would like to thank Prof.~Alexander Barg for suggesting this problem, providing valuable guidance, and offering meticulous proofreading. Additionally, the author extends thanks to Wen-Tai Hsu, Yihan Zhang, and Shaoyang Zhou for useful discussions. 

\bibliographystyle{plain}
\bibliography{main.bib}
\end{document}